\title{Degree powers in graphs with a forbidden forest}

\author{Yongxin Lan$^1$, Henry Liu$^2$, Zhongmei Qin$^3$, Yongtang Shi$^1$\footnote{Corresponding author}\\
\\
\normalsize $^1$Center for Combinatorics and LPMC\\
\normalsize Nankai University, Tianjin 300071, China\\
\normalsize yxlan0@126.com, shi@nankai.edu.cn\\
\\
\normalsize $^2$School of Mathematics\\
\normalsize Sun Yat-sen University, Guangzhou 510275, China\\
\normalsize liaozhx5@mail.sysu.edu.cn\\
\\
\normalsize $^3$College of Science\\
\normalsize Chang'an University, Xi'an, Shaanxi 710064, China\\
\normalsize qinzhongmei90@163.com\\
}
\date{6 January 2018}

\documentclass[11pt]{article}
\usepackage{amsmath,amssymb,amscd,amsthm}
\usepackage{xcolor}
\usepackage{multirow}
\newdimen\unit\newdimen\psep\newcount\nd\newcount\ndx\newbox\dotb\newbox\ptbox
\newdimen\dx\newdimen\dy\newdimen\dxx\newdimen\dyy\newdimen\hgt
\newdimen\xoff\newdimen\yoff
\newcommand\clap[1]{\hbox to 0pt{\hss{#1}\hss}}
\newcommand\vdisk[1]{{\font\dotf=cmr10 scaled #1\dotf.}}
\newcommand\varline[2]{\setbox\dotb\hbox{\vdisk{#1}}\xoff=-.5\wd\dotb
\wd\dotb=0pt\yoff=-.5\ht\dotb\psep=#2\ht\dotb}
\newcommand\varpt[1]{\setbox\ptbox\clap{\vdisk{#1}}\setbox\ptbox
\hbox{\raise-.5\ht\ptbox\box\ptbox}}
\newcommand\cpt{\copy\ptbox}
\newcommand\point[3]{\rlap{\kern#1\unit\raise#2\unit\hbox{#3}}}
\newcommand\setnd[4]{\dx=#3\unit\advance\dx-#1\unit\divide\dx by\psep
\dy=#4\unit\advance\dy-#2\unit\divide\dy by\psep \multiply\dx
by\dx\multiply\dy by\dy\advance\dx\dy\nd=1\advance\dx-1sp
\loop\ifnum\dx>0\advance\dx-\nd sp\advance\nd1\advance\dx-\nd
sp\repeat}
\newcommand\dl[4]{{\setnd{#1}{#2}{#3}{#4}\dline{#1}{#2}{#3}{#4}\nd}}
\newcommand\dline[5]{{\nd=#5\hgt=#2\unit\dx=#3\unit\advance\dx-#1\unit
\divide\dx by\nd\dy=#4\unit\advance\dy-#2\unit\divide\dy by\nd
\advance\hgt\yoff\rlap{\kern#1\unit\kern\xoff\loop\ifnum\nd>1\advance\nd-1
\advance\hgt\dy\kern\dx\raise\hgt\copy\dotb\repeat}}}
\newcommand\ellipse[4]{\qellip{#1}{#2}{#3}{#4}\qellip{#1}{#2}{#3}{-#4}%
\qellip{#1}{#2}{-#3}{#4}\qellip{#1}{#2}{-#3}{-#4}}
\newcommand\qellip[4]{{\setnd{0}{0}{#3}{#4}\dx=\unit\dy=0pt\raise\yoff\rlap{%
\kern#1\unit\kern\xoff\raise#2\unit\hbox{\loop\ifnum\dx>0\rlap{\kern#3\dx
\raise#4\dy\copy\dotb}\hgt=\dx\divide\hgt
by\nd\advance\dy\hgt\hgt=\dy \divide\hgt
by\nd\advance\dx-\hgt\repeat\rlap{\raise#4\dy\copy\dotb}}}}}
\newcommand\bez[6]{{\setnd{#1}{#2}{#3}{#4}\ndx=\nd\setnd{#3}{#4}{#5}{#6}
\ifnum\ndx>\nd\nd=\ndx\fi\dx=#3\unit\advance\dx-#1\unit\dy=#4\unit
\advance\dy-#2\unit\dxx=#5\unit\advance\dxx-#1\unit\dyy=#6\unit\advance
\dyy-#2\unit\advance\dxx-2\dx\advance\dyy-2\dy\divide\dxx
by\nd\divide\dyy
by\nd\advance\dx.25\dxx\advance\dy.25\dyy\divide\dx
by\nd\divide\dy by\nd \multiply\nd
by2\dx=100\dx\dy=100\dy\dxx=100\dxx\dyy=100\dyy\divide\dxx by\nd
\divide\dyy
by\nd\hgt=#2\unit\raise\yoff\rlap{\kern#1\unit\kern\xoff
\raise\hgt\copy\dotb\loop\ifnum\nd>0\advance\nd-1\advance\hgt0.01\dy
\kern0.01\dx\raise\hgt\copy\dotb\advance\dx\dxx\advance\dy\dyy\repeat}}}
\newcommand\ptu[3]{\point{#1}{#2}{\cpt\raise1ex\clap{$\scriptstyle{#3}$}}}
\newcommand\ptd[3]{\point{#1}{#2}{\cpt\raise-1.8ex\clap{$\scriptstyle{#3}$}}}
\newcommand\ptr[3]{\point{#1}{#2}{\cpt\raise-.4ex\rlap{$\ \scriptstyle{#3}$}}}
\newcommand\ptl[3]{\point{#1}{#2}{\cpt\raise-.4ex\llap{$\scriptstyle{#3}\ $}}}
\newcommand\ptlu[3]{\point{#1}{#2}{\raise.8ex\clap{$\scriptstyle{#3}$}}}
\newcommand\ptld[3]{\point{#1}{#2}{\raise-1.6ex\clap{$\scriptstyle{#3}$}}}
\newcommand\ptlr[3]{\point{#1}{#2}{\raise-.4ex\rlap{$\,\scriptstyle{#3}$}}}
\newcommand\ptll[3]{\point{#1}{#2}{\raise-.4ex\llap{$\scriptstyle{#3}\,$}}}
\newcommand\pt[2]{\point{#1}{#2}{\cpt}}

\newcommand\thnline{\varline{400}{.6}}

\varpt{2500}\thnline\unit=2em

\newtheorem{theorem}                   {Theorem}%[section]
\newtheorem{thm}             [theorem] {Theorem}
\newtheorem{lemma}           [theorem] {Lemma}
\newtheorem{lem}             [theorem] {Lemma}

\newtheorem{cor}             [theorem] {Corollary}

\newtheorem{prop}             [theorem] {Proposition}
\newtheorem{claim}           [theorem] {Claim}

\newtheorem{conj}      [theorem] {Conjecture}

\newtheorem*{clm1'}{Claim 1$'$}
\newtheorem*{clm2'}{Claim 2$'$}
\newtheorem*{clm3'}{Claim 3$'$}
\newtheorem*{clm4'}{Claim 4$'$}

\def\eps{\varepsilon}
\def\ex{\textup{ex}}

\begin{document}
\maketitle

\begin{abstract}
\noindent Given a positive integer $p$ and a graph $G$ with degree sequence $d_1,\ldots,d_n$, we define $e_p(G)=\sum_{i=1}^n d_i^p$. Caro and Yuster introduced a Tur\'an-type problem for $e_p(G)$: Given a positive 
integer $p$ and a graph $H$, determine the function $\textup{ex}_p(n,H)$, which is the maximum value of $e_p(G)$ taken over all graphs $G$ on $n$ vertices that do not contain $H$ as a subgraph. Clearly, $\textup{ex}_1(n,H)=2\textup{ex}(n,H)$, where $\textup{ex}(n,H)$ denotes the classical Tur\'an number. Caro and Yuster determined the function $\textup{ex}_p(n, P_\ell)$ for sufficiently large $n$, where $p\geq 2$ and $P_\ell$ denotes the path on $\ell$ vertices. In this paper, we generalise this result and determine $\textup{ex}_p(n,F)$ for  sufficiently large $n$, where $p\geq 2$ and $F$ is a linear forest. We also determine $\textup{ex}_p(n,S)$, where $S$ is a star forest; and $\textup{ex}_p(n,B)$, where $B$ is a broom graph with diameter at most six.
\\[2mm]
[2mm] \textbf{AMS Subject Classification (2010):} 05C07, 05C35
\end{abstract}

\section{Introduction}
For standard graph-theoretic notation and terminology, the reader is referred to \cite{B98}. All graphs considered here are finite, undirected, and have no loops or multiple edges. Let $G$ and $H$ be two graphs. The degree of a vertex $v\in V(G)$ and the maximum degree of $G$ are denoted by $d_G(v)$ and $\Delta(G)$. We use $G\cup H$ to denote the disjoint union of $G$ and $H$, and $G +H$ for the \emph{join} of $G$ and $H$, i.e., the graph obtained from $G \cup H$ by adding all edges between $G$ and $H$. Let $kG$ denote $k$ vertex-disjoint copies of $G$. For $U\subset V(G)$, let $G[U]$ denote the subgraph of $G$ induced by $U$. Let $K_t, E_t$ and $P_t$ denote the complete graph, the empty graph, and the path on $t$ vertices, respectively. Let $S_r$ denote the star with maximum degree $r$. Let $M_t$ be the graph on $t$ vertices with a maximum matching (i.e., $\lfloor \frac{t}{2}\rfloor$ independent edges). 

Given a graph $H$, we say that a graph $G$ is \emph{$H$-free} if $G$ does not contain $H$ as a subgraph. The classical \emph{Tur\'{a}n number}, denote by $\textup{ex}(n,H)$, is the maximum number of edges in a $H$-free graph on $n$ vertices. Tur\'an's classical result \cite{T41} states that $\textup{ex}(n,K_{r+1})=e(T_r(n))$ for $n\ge r\ge 2$, where $T_r(n)$ denotes the \emph{$r$-partite Tur\'{a}n graph} on $n$ vertices. Given a graph $G$ whose degree sequence is $d_1,\ldots,d_n$, and a positive integer $p$, let $e_p(G)=\sum_{i=1}^n {d_i^p}$. Caro and Yuster \cite{CY00} introduced a Tur\'an-type problem for $e_p(G)$: Determine the function $\textup{ex}_p(n,H)$, which is the maximum value of $e_p(G)$ taken over all $H$-free graphs $G$ on $n$ vertices. Moreover, characterise the \emph{extremal graphs}, i.e., the $H$-free graphs $G$ on $n$ vertices with $e_p(G)=\textup{ex}_p(n,H)$. Clearly, we have $\textup{ex}_1(n,H)=2\textup{ex}(n,H)$. 

This Tur\'an-type problem has attracted significant interest from many researchers. Caro and Yuster \cite{CY00} proved that $\textup{ex}_p(n,K_{r+1})=e_p(T_r(n))$ for $p=1,2,3$. The same result does not hold if $r$ is fixed, and $p$ and $n$ are sufficiently large. For example, if $G$ is the complete bipartite graph with class sizes $\lfloor\frac{n}{2}\rfloor-1$ and $\lceil\frac{n}{2}\rceil+1$, then we have $e_4(G)>e_4(T_2(n))$. Hence, we see that the parameter $p$ does play a role in the value of $\textup{ex}_p(n,K_{r+1})$ and the extremal graphs. Bollob\'as and Nikiforov further studied the function $\textup{ex}_p(n,K_{r + 1})$, where they allowed $p>0$ to be real. In \cite{BN04}, they proved that for $n$ sufficiently large, $\textup{ex}_p(n, K_{r+1})=e_p(T_r(n))$ for $0<p<r$, and $\textup{ex}_p(n, K_{r+1})>(1+\eps)e_p(T_r(n))$ for $p\ge r+\lceil\sqrt{2r}\rceil$ and some $\eps=\eps(r)>0$. In \cite{BN12}, they proved a result which gives an extension of the Erd\H{o}s-Stone Theorem by using $e_p(G)$ instead of the number of edges.

When considering cycles as the forbidden subgraphs, Caro and Yuster \cite{CY00} proved that $\textup{ex}_2(n,\mathcal C)=e_2(F_n)$ for sufficiently large $n$, where $\mathcal C$ denotes the family of cycles with even length (notice the
natural extension of the definition of $\ex_p$ to families of graphs), and $F_n$ is the \emph{friendship graph} on $n$ vertices, i.e., $F_n$ is obtained by taking a star on $n$ vertices and adding a maximum matching on the set of leaves. They also showed that $F_n$ is the unique extremal graph, and remarked that the same result also holds for $p>2$. Nikiforov \cite{N09} proved that $\textup{ex}_p(n,C_{2k+2}) = (1 + o(1))kn^p$, where $C_t$ denotes the cycle of order $t$, and this settled a conjecture of Caro and Yuster. Gu et al.~\cite{GLS15} proved that for $p\ge 1$, there exists a constant $c = c(p)$ such that the following holds: If $\textup{ex}_p(n,C_5) = e_p(G)$ for some $C_5$-free graph $G$ of order $n$, then $G$ is a complete bipartite graph with class sizes $cn + o(n)$ and $(1 - c)n + o(n)$.

%A \emph{broom} is a graph which is obtained by identifying the centre of a star with an end-vertex of a path. For $\ell\ge 2$ and $s\ge 0$, let $B_{\ell,s}$ denote the broom on $\ell+s$ vertices by taking a path $v_1\cdots v_\ell$ and attaching $s$ pendent vertices to $v_2$. We assume in addition that if $\ell=2$, then $s=0$ (If $\ell=2$ and $s\ge 1$, then $B_{\ell,s}=B_{2,s}$ would be the same as $B_{3,s-1}$). A \emph{broom forest} is a forest whose connected components are brooms. Note that paths and stars are special types of broom. A \emph{linear forest} (resp.~\emph{star forest}) is a forest whose connected components are paths (resp.~stars). In this paper, our main aim will be to determine the function $\textup{ex}_p(n,F)$, where $p\ge 2$, $F$ is a broom forest, and $n$ is sufficiently large. Our main results are Theorems \ref{SFthm} and \ref{broomforestthm}, where we determine $\textup{ex}_p(n,F)$ in the cases when $F$ is a star forest and not a star forest, respectively. We will also determine the corresponding extremal graphs.

A \emph{linear forest} (resp.~\emph{star forest}) is a forest whose connected components are paths (resp.~stars). There are many known results about the function $\textup{ex}_p(n,F)$ where $F$ is a linear forest. For the case of the classical Tur\'an number $\textup{ex}(n,F)$, one of the earliest results is the case when $F=P_\ell$ is a path. Erd\H{o}s and Gallai \cite{EG59} proved in 1959 that $\textup{ex}(n,P_\ell)\leq (\frac{\ell}{2}-1)n$ for $\ell\ge 2$, and if $\ell-1$ divides $n$, then equality holds only for the graph with vertex-disjoint copies of $K_{\ell-1}$. Motivated by this result, Erd\H{o}s and S\'os \cite{E64} in 1963 made the conjecture that the same result holds for any tree, i.e., if $T$ is a tree on $t\ge 2$ vertices, then we have $\textup{ex}(n,T)\leq (\frac{t}{2}-1)n$. This long-standing conjecture remains open, and many partial results are known. %The case when $T$ is a broom is known to be true, and is a result proved by Sacl\'e which was mentioned by Wo\'zniak in \cite{W96}. 
The result of Erd\H{o}s and Gallai was also sharpened by Faudree and Schelp \cite{FS75}, when they determined the function $\textup{ex}(n,P_\ell)$ exactly as well as the extremal graphs. When $F$ has more components, Erd\H{o}s and Gallai \cite{EG59} also proved that $\textup{ex}(n,kP_2)={k-1\choose 2}+(k-1)(n-k+1)$ for $k\ge 2$ and sufficiently large $n$, where the unique extremal graph is $K_{k-1}+E_{n-k+1}$. Very recently, this result was extended by Bushaw and Kettle \cite{BK11}, who determined the function $\textup{ex}(n,kP_\ell)$ for $k\ge 2$, $\ell\ge 3$ and sufficiently large $n$. Their result was further generalised by Lidick\'y et al.~\cite{LLP13}, who determined the function $\textup{ex}(n,F)$ for an arbitrary linear forest $F$ and sufficiently large $n$. In these two results, the extremal graph is unique. Lidick\'y et al.~\cite{LLP13} also determined the function $\textup{ex}(n,S)$ for an arbitrary star forest $S$ and sufficiently large $n$, and characterised the extremal graphs.

On the other hand, Caro and Yuster \cite{CY00} determined the function $\textup{ex}_p(n, P_\ell)$ for $p\geq 2$, $\ell\ge 3$ and sufficiently large $n$. The extremal graph is again unique, and is significantly different to the extremal graphs of $\textup{ex}(n,P_\ell)$ obtained by Faudree and Schelp \cite{FS75}. They also determined the functions $\textup{ex}_p(n,S_r)$ and $\textup{ex}_p(n,S_r^\ast)$, and their extremal graphs, where $S_r^\ast$ is the graph obtained by attaching a pendent edge at a leaf of $S_r$. %They also determined the functions $\textup{ex}_p(n,S_r)$ and $\textup{ex}_p(n,B_{4,s})$ and their extremal graphs.

This paper will be organised as follows. In Section \ref{surveysect}, we will state precisely the previously known results about the function $\textup{ex}_p(n,F)$, for various forests $F$. %In Section \ref{broomforestsect}, we will prove Theorems \ref{SFthm} and \ref{broomforestthm}. In Section \ref{turanbroomforestsect}, we will study the Tur\'an function $\textup{ex}(n,F)$ where $F$ is a broom forest will some restrictions. 
In Sections \ref{linstarforestsect} and \ref{broomsect}, we will determine the function $\ex_p(n,F)$ when $F$ is a linear forest, a star forest, and a broom with diameter at most $6$ (A \emph{broom} is a path with a star attached at one end). Our results can be regarded as extensions to many of these previously known results from \cite{BK11,CY00,EG59,LLP13}. Unless otherwise stated, we assume that $n$ is always sufficiently large, and we will make no serious attempt to minimise the lower bound on $n$. Without going into details, we remark that every large lower bound on $n$ depends only on the forest $F$, and not the parameter $p$.

\section{Known results}\label{surveysect}

%In this section, we will review many of the known results about the function $\textup{ex}_p(n,F)$, where $F$ is a broom forest. Since there are now many known results, we believe that this will serve as an interesting survey to the reader. This section also provides motivation for our results in Sections \ref{broomforestsect} and \ref{turanbroomforestsect}, as well as allowing us to collect some tools for our proofs later.

In this section, we will review many of the known results about the function $\textup{ex}_p(n,F)$, for various forests $F$. Some of these results will also be helpful for us to present our results in Sections \ref{linstarforestsect} and \ref{broomsect}. First, we collect the results where $F$ is a single component. When $F$ is a path, Caro and Yuster \cite{CY00} observed that for $p\ge 1$, we have
\begin{equation}\label{CYobs1}
\textup{ex}_p(n,P_2) = 0,\quad \textup{and} \quad
\textup{ex}_{p}(n,P_{3})=
\left\{ 
\begin{array}{ll}
n-1 & \textrm{if $n$ is odd,}\\
n& \textrm{if $n$ is even.}
\end{array} 
\right.
\end{equation}
Moreover, the unique extremal graph for $\textup{ex}_{p}(n,P_{3})$ is $M_n$, the graph on $n$ vertices with a maximum matching. For $F=P_\ell$, Erd\H{o}s and Gallai \cite{EG59} proved the following result, as we have mentioned in the introduction.

\begin{thm}\label{EGthm1}\,\textup{\cite{EG59}}
For $\ell\ge 2$, we have $\textup{ex}(n, P_\ell) \le (\frac{\ell}{2}-1)n$. Moreover, if $\ell-1$ divides $n$, then equality holds only for the graph with vertex-disjoint copies of $K_{\ell-1}$.
\end{thm}

Inspired by Theorem \ref{EGthm1}, Erd\H{o}s and S\'os \cite{E64} made the conjecture that the same result holds for any tree: If $T$ is a tree on $t\ge 2$ vertices, then $\textup{ex}(n,T)\leq (\frac{t}{2}-1)n$. This long-standing conjecture remains open, and many partial results are known. Theorem \ref{EGthm1} was subsequently sharpened by Faudree and Schelp \cite{FS75}, when they managed to determine $\textup{ex}(n,P_\ell)$ exactly, as well as all the extremal graphs.

\begin{thm}\label{FSthm}\,\textup{\cite{FS75}}
Let $\ell\ge 2$ and $n=a(\ell-1)+b$, where $a\ge 0$ and $0\le b<\ell-1$. We have
\[
\textup{ex}(n, P_\ell)=a{\ell-1 \choose 2}+{b\choose 2}.
\]
Moreover, the extremal graphs are: 
\begin{itemize}
\item $aK_{\ell-1}\cup K_b$, 
\item $a'K_{\ell-1}\cup \big(K_{\ell/2-1}+E_{\ell/2+(a-a'-1)(\ell-1)+b}\big)$, where $\ell$ is even, $a>0$, $b=\frac{\ell}{2}$ or $b=\frac{\ell}{2}-1$ and $0\le a'<a$.
\end{itemize}
\end{thm}

%Inspired by Theorem \ref{EGthm1}, Erd\H{o}s and S\'os \cite{E64} made the conjecture that the same result holds for any tree: If $T$ is a tree on $t\ge 2$ vertices, then $\textup{ex}(n,T)\leq (\frac{t}{2}-1)n$. This long-standing conjecture remains open, and many partial results are known. The case when $T$ is a broom is a result of Sacl\'e, as follows.

%\begin{thm}[Sacl\'e, mentioned in \cite{W96}]\label{saclethm}
%Let $B$ be a broom on $t\ge 2$ vertices. We have $\textup{ex}(n, B) \le (\frac{t}{2}-1)n$.
%\end{thm}

Caro and Yuster \cite{CY00} determined the function $\textup{ex}_p(n,P_\ell)$ for $p\ge 2$, $\ell\ge 4$, and sufficiently large $n$, and they showed that the extremal graph is unique. To state their result, we define the graph $H(n,\ell)$ 
%for $n\ge \ell\ge 4$ 
as follows. Let $b=\lfloor\frac{\ell}{2}\rfloor-1$. Then $H(n,\ell)=K_b+E_{n-b}$ if $\ell$ is even, and $H(n,\ell)$ is $K_b+E_{n-b}$ with an edge added to $E_{n-b}$ if $\ell$ is odd.

\begin{thm}\label{CYthm1}\,\textup{\cite{CY00}}
Let $p\ge 2$, $\ell\ge 4$, and $n \ge n_0(\ell)$ be sufficiently large. Then
\begin{align*}
\textup{ex}_p(n, P_\ell) &= e_p(H(n,\ell))\\
&=\left\{
\begin{array}{ll}
b(n-1)^{p}+(n-b-2)b^{p}+2(b+1)^p & \textup{\emph{if $\ell$ is odd,}}\\
b(n-1)^{p}+(n-b)b^p & \textup{\emph{if $\ell$ is even,}}
\end{array}
\right.
\end{align*}
where $b=\lfloor\frac{\ell}{2}\rfloor-1$. Moreover, $H(n,\ell)$ is the unique extremal graph.
\end{thm}

They remarked that the extremal graph $H(n,\ell)$ for $\textup{ex}_p(n,P_\ell)$, with $p\ge 2$, is very different from the extremal graphs for $\textup{ex}(n,P_\ell)$ in Theorem \ref{FSthm}. This is because $H(n,\ell)$ has large maximum degree, which plays a role in making the value of $e_p(H(n,\ell))$ large, when $p\ge 2$.

When $F=S_r$ is a star, Caro and Yuster \cite{CY00} made the observation that $\textup{ex}_p(n,S_r)$ is attained by a graph $L$ on $n$ vertices which is an extremal graph for $\textup{ex}(n,S_r)$. Clearly if $n\le r-1$, we have $L=K_n$. For $n\ge r$, we have $L$ is an $(r-1)$-regular graph if $(r-1)n$ is even, and $L$ has $n-1$ vertices of degree $r-1$ and one vertex of degree $r-2$ if $(r-1)n$ is odd. We call such a graph $L$ a \emph{near $(r-1)$-regular graph}, since $L$ is as close to being $(r-1)$-regular as possible. It is well-known and easy to show that such graphs $L$ exist. Note that we have $e(L)=\big\lfloor\frac{(r-1)n}{2}\big\rfloor$. Thus, the observation of Caro and Yuster is the following.

\begin{prop}\label{CYprp1}\,\textup{\cite{CY00}}
Let $p\ge 1$, and let $S_r$ be the star with maximum degree $r\ge 1$.
\begin{enumerate}
\item[(a)] If $n\leq r-1$, then $\textup{ex}_{p}(n,S_r)=n(n-1)^{p}$. Moreover, the unique extremal graph is $K_n$.
\item[(b)] If $n\ge r$, then
\[
\ex_p(n,S_r)=
\left\{
\begin{array}{ll}
(n-1)(r-1)^{p}+(r-2)^{p} & \textup{\emph{if $(r-1)n$ is odd,}}\\
n(r-1)^{p} & \textup{\emph{if $(r-1)n$ is even.}}
\end{array}
\right.
\]
Moreover, the extremal graphs are the near $(r-1)$-regular graphs on $n$ vertices.
\end{enumerate}
\end{prop}

For $\ell\ge 4$ and $s\ge 0$, let $B_{\ell,s}$ be the graph on $\ell+s$ vertices, obtained by adding $s$ pendent edges to a penultimate vertex $v$ of $P_\ell$. Such a graph $B_{\ell,s}$ is a \emph{broom}, and $v$ is the \emph{centre} of the broom.\\[1ex]
\[ \unit = 1cm
\thnline 
\dl{-3}{0}{-1}{0}\dl{0}{0}{3}{0}\dl{2}{0}{1.35}{1}\dl{2}{0}{1.6}{1}\dl{2}{0}{1.85}{1}\dl{2}{0}{2.65}{1}
\pt{-3}{0}\pt{-2}{0}\pt{-1}{0}\pt{3}{0}\pt{2}{0}\pt{1}{0}\pt{0}{0}
\point{1.6}{0.1}{{\footnotesize$v$}}\point{-0.65}{-0.8}{{\footnotesize$\ell$ vertices}}\point{1.34}{1.58}{{\footnotesize$s$ vertices}}\point{-0.73}{-0.1}{$\cdots$}
\dl{-3}{-0.35}{-0.1}{-0.35}\dl{3}{-0.35}{0.1}{-0.35}\bez{-3}{-0.35}{-3.1}{-0.35}{-3.1}{-0.25}\bez{3}{-0.35}{3.1}{-0.35}{3.1}{-0.25}\bez{-0.1}{-0.35}{0}{-0.35}{0}{-0.45}\bez{0.1}{-0.35}{0}{-0.35}{0}{-0.45}
\point{2.03}{0.9}{$\cdots$}\pt{1.35}{1}\pt{1.6}{1}\pt{1.85}{1}\pt{2.65}{1}
\dl{1.35}{1.35}{1.9}{1.35}\dl{2.1}{1.35}{2.65}{1.35}\bez{1.35}{1.35}{1.25}{1.35}{1.25}{1.25}\bez{2.65}{1.35}{2.75}{1.35}{2.75}{1.25}\bez{1.9}{1.35}{2}{1.35}{2}{1.45}\bez{2.1}{1.35}{2}{1.35}{2}{1.45}
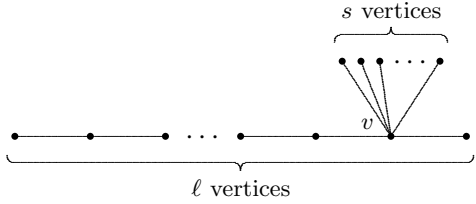
\ptlu{0}{-1.8}{\textup{\normalsize Figure 1. The broom graph $B_{\ell,s}$}}
\]\\[0.7ex]
\indent It is interesting to study Tur\'an-type problems for brooms, because a broom may be considered as a generalisation of both a path and a star. Sun and Wang \cite{SW11} determined the function $\ex(n, B_{4,s})$ for $s\ge 1$, as follows.

\begin{thm}\label{SWthm1}\,\textup{\cite{SW11}}
Let $s\ge 1$ and $n\ge s+4$. Let $n=a(s+3)+b$, where $a\ge 1$ and $0\le b<s+3$. We have
\[
\ex(n, B_{4,s})=
\left\{
\begin{array}{ll}
\displaystyle (a-1){s+3\choose 2}+\Big\lfloor\frac{(s+1)(s+3+b)}{2}\Big\rfloor & \textup{\emph{if $s\ge 3$ and $2\le b\le s$,}}\\[2ex]
\displaystyle a{s+3\choose 2}+{b\choose 2}  & \textup{\emph{otherwise.}}
\end{array}
\right.
\]
\end{thm}

Roughly speaking, in Theorem \ref{SWthm1}, the value of $\ex(n,B_{4,s})$ is attained as follows. If $b$ is close to either $0$ or $s+3$, then we would take the graph $aK_{s+3}\cup K_b$. Otherwise, we would take a graph $(a-1)K_{s+3}\cup L$, where $L$ is a near $(s+1)$-regular graph on $s+3+b$ vertices.

Sun and Wang also determined the function $\ex(n,B_{5,s})$ for $s\ge 1$ and $n\ge s+5$. However, their result is complicated to state in full. A key result that they proved is the following.
\begin{thm}\label{SWthm2}\,\textup{\cite{SW11}}
Let $s\ge 1$ and $n\ge s+5$. Let $n=a(s+4)+b$, where $a\ge 1$ and $0\le b<s+4$. We have
\[
\ex(n, B_{5,s})=
\left\{
\begin{array}{ll}
\displaystyle (a-1){s+4\choose 2}+\ex(s+4+b, B_{5,s}) & \textup{\emph{if $1\le b\le s$,}}\\[2ex]
\displaystyle a{s+4\choose 2}+{b\choose 2}  & \textup{\emph{if $b\in\{0,s+1,s+2,s+3\}$.}}
\end{array}
\right.
\]
\end{thm}

Similarly, in Theorem \ref{SWthm2}, the value of $\ex(n,B_{5,s})$ is attained by $aK_{s+4}\cup K_b$ if $b$ is either $0$ or close to $s+4$. Otherwise, we would take a graph $(a-1)K_{s+4}\cup L$, where $L$ is an extremal graph for $B_{5,s}$ on  $s+4+b$ vertices.

Caro and Yuster \cite{CY00} determined the function $\ex_p(n,B_{4,s})$, for $p\ge 2$ and sufficiently large $n$. They remarked that the result is very different to Proposition \ref{CYprp1}, even though $B_{4,s}$ is very close to being a star.
\begin{prop}\label{CYprp2}\,\textup{\cite{CY00}}
Let $p\ge 2$, $s\ge 1$, and $n > 2(s+4)$. Then $\textup{ex}_{p}(n, B_{4,s}) = e_p(S_{n-1}) = (n - 1)^p + (n - 1)$. Moreover, $S_{n-1}$ is the unique extremal graph.
\end{prop}
%
%***$ex(n,B_{4,s})$, $ex(n,B_{5,s})$; $ex(n,B_{\ell,s})$ seems not simple.***
%
%In addition, Caro and Yuster also determined $\textup{ex}_p(n,B_{4,s})$, as follows.
%
%\begin{prop}\label{CYprp2}\,\textup{\cite{CY00}}
%Let $p\ge 2$, $s\ge 0$, and $n > 2(s+4)$. Then $\textup{ex}_p(n, B_{4,s}) = e_p(S_n) = (n - 1)^p + n - 1$.
%\end{prop}
%
%They noted that, although $B_{4,s}$ is very close to being a star, the result for $\textup{ex}_p(n, B_{4,s})$ is quite different from the result for $\textup{ex}_p(n, S_r)$.

Now we consider the case when the forest $F$ has more than one component. When $F=kP_2$, the classical Tur\'an number $\textup{ex}(n,kP_2)$ was determined by Erd\H{o}s and Gallai \cite{EG59}.

%\begin{thm}\label{EGthm2}\,\textup{\cite{EG59}}
%Let $k\ge 2$. We have
%\[
%\textup{ex}(n,kP_2)=
%\left\{ 
%\begin{array}{ll}
%\displaystyle {n\choose 2} & \textrm{if $n<2k$,}\\[2.5ex]
%\displaystyle {2k-1\choose 2} & \textrm{if $2k\le n<\frac{5k}{2}-1$,}\\[2.5ex]
%\displaystyle {k-1\choose 2}+(k-1)(n-k+1) & \textrm{if $n\ge \frac{5k}{2}-1$.}
%\end{array} 
%\right.
%\]
%Moreover, the extremal graphs are:
%\begin{itemize}
%\item $K_n$ if $n<2k$,
%\item $K_{2k-1}\cup E_{n-2k+1}$ if $2k\le n<\frac{5k}{2}-1$
%\item $K_{k-1}+E_{n-k+1}$ if $n>\frac{5k}{2}-1$,
%\item $K_{2k-1}\cup E_{k/2}$ and $K_{k-1}+E_{3k/2}$ if $n=\frac{5k}{2}-1$ with $k$ even.
%\end{itemize}
%\end{thm}

\begin{thm}\label{EGthm2}\,\textup{\cite{EG59}}
Let $k\ge 2$ and $n> \frac{5k}{2}-1$. We have $\ex(n,kP_2)={k-1\choose 2}+(k-1)(n-k+1)$. Moreover, $K_{k-1}+E_{n-k+1}$ is the unique extremal graph.
\end{thm}

For $n\le \frac{5k}{2}-1$, Erd\H{o}s and Gallai also determined $\ex(n,kP_2)$ and the extremal graphs, which are different from those in Theorem \ref{EGthm2}. For the function $\textup{ex}(n,kP_3)$, Yuan and Zhang \cite{YZ17} obtained the following result.

\begin{thm}\label{YZthm}\,\textup{\cite{YZ17}}
Let $k\ge 2$ and $n>5k-1$. We have $\ex(n,kP_3)={k-1 \choose 2}+(k-1)(n-k+1)+\lfloor \frac{n-k+1}{2}\rfloor$.
Moreover, $K_{k-1}+M_{n-k+1}$ is the unique extremal graph.
\end{thm}

In fact, Yuan and Zhang completely determined $\ex(n,kP_3)$ and the extremal graphs for all $n$, which solved a conjecture of Gorgol \cite{G11}. Bushaw and Kettle \cite{BK11} had previously proved the case of Theorem \ref{YZthm} for $n\ge 7k$.
%
%Gorgol \cite{G11} made a conjecture for the exact answer. Her conjecture was very recently solved by Yuan and Zhang \cite{YZ17}, as follows.
%\begin{thm}\label{YZthm}\,\textup{\cite{YZ17}}
%Let $k\ge 2$. We have
%\[
%\textup{ex}(n,kP_3)=
%\left\{ 
%\begin{array}{ll}
%\displaystyle {n\choose 2} & \textrm{if $n<3k$,}\\[2.5ex]
%\displaystyle {3k-1\choose 2}+\Big\lfloor \frac{n-3k+1}{2}\Big\rfloor & \textrm{if $3k\le n<5k-1$,}\\[2.5ex]
%\displaystyle {k-1 \choose 2}+(k-1)(n-k+1)+\Big\lfloor \frac{n-k+1}{2}\Big\rfloor & \textrm{if $n\ge 5k-1$.}
%\end{array} 
%\right.
%\]
%Moreover, the extremal graphs are: 
%\begin{itemize}
%\item $K_n$ if $n<3k$,
%\item $K_{3k-1}\cup M_{n-3k+1}$ if $3k\le n<5k-1$,
%\item $K_{k-1}+M_{n-k+1}$ if $n>5k-1$,
%\item $K_{3k-1}\cup M_{2k}$ and $K_{k-1}+M_{4k}$ if $n=5k-1$.
%\end{itemize}
%\end{thm}

Next, there are results for the case when $F=\bigcup_{i=1}^kP_{\ell_i}$ is a linear forest, where $k\ge 2$, and we may assume that $\ell_1\ge \cdots \ge \ell_k\ge 2$. To describe the results, we define the graph $H(n,F)$ as follows. Let $b=\sum_{i=1}^{k}\lfloor\frac{\ell_{i}}{2}\rfloor-1$. Then, $H(n,F)$ is $K_b+E_{n-b}$ with a single edge added to $E_{n-b}$ if all $\ell_{i}$ are odd, and $H(n,F)=K_b+E_{n-b}$ otherwise. Note that $H(n,F)$ is $F$-free. Indeed, if $H(n,F)$ contains $F$, then the path in $F$ of order $\ell_i$ must use at least $\lfloor\frac{\ell_i}{2}\rfloor$ vertices of the $K_b$. But this cannot happen for every path in $F$, by the definition of $b$.

In the case when $F=kP_\ell$, we write $H(n,k,\ell)$ for $H(n,F)$. We have already seen the results for $\textup{ex}(n,kP_\ell)$ when $\ell=2,3$ (Theorems \ref{EGthm2} and \ref{YZthm}). For $\ell\ge 4$, Bushaw and Kettle \cite{BK11} proved the following result.

\begin{thm}\label{BKthm1}\,\textup{\cite{BK11}}
Let $k\geq 2$, $\ell\geq 4$, and $n\ge 2\ell+2k\ell(\lceil\frac{\ell}{2}\rceil+1)\binom{\ell}{\lfloor\frac{\ell}{2}\rfloor}$. We have
\[
\textup{ex}(n,kP_{\ell}) = e(H(n,k,\ell))={k\lfloor\frac{\ell}{2}\rfloor-1 \choose 2}+\Big(k\Big\lfloor\frac{\ell}{2}\Big\rfloor-1\Big)\Big(n-k\Big\lfloor\frac{\ell}{2}\Big\rfloor+1\Big)+c,
\]
where $c=1$ if $\ell$ is odd, and $c=0$ if $\ell$ is even. Moreover, $H(n,k,\ell)$ is the unique extremal graph. 
\end{thm}

This result was extended by Lidick\'y et al.~\cite{LLP13}, who determined $\textup{ex}(n, F)$ for an arbitrary linear forest $F\neq kP_3$.

\begin{thm}\label{LLPthm1}\,\textup{\cite{LLP13}}
Let $k\ge 2$, and $F=\bigcup_{i=1}^kP_{\ell_i}$ be a linear forest, where $\ell_{1}\ge \ell_{2}\ge\cdots\ge \ell_{k}\ge 2$ and $\ell_{i}\neq 3$ for some $i$. Let $n\ge n_0(F)$ be sufficiently large. We have
\[
\textup{ex}(n,F)=e(H(n,F))={\sum_{i=1}^{k}\lfloor\frac{\ell_{i}}{2}\rfloor-1 \choose 2}+\Big(\sum_{i=1}^{k}\Big\lfloor\frac{\ell_{i}}{2}\Big\rfloor-1\Big)\Big(n-\sum_{i=1}^{k}\Big\lfloor\frac{\ell_{i}}{2}\Big\rfloor+1\Big)+c,
\]
where $c=1$ if all $\ell_{i}$ are odd, and $c=0$ otherwise. Moreover, $H(n,F)$ is the unique extremal graph.
\end{thm}

Finally, Lidick\'y et al.~\cite{LLP13} determined the function $\textup{ex}(n,F)$, when $F$ is a star forest and $n$ is sufficiently large. Let $F=\bigcup_{i=1}^kS_{r_i}$, where $r_{1}\geq \cdots \geq r_{k}\ge 1$. To describe their result, we define a graph $G(n,F)$ as follows. Let $i,r\ge 1$, and $L$ be a graph on $n-i+1$ vertices which is an extremal graph for $S_{r}$. Thus $L$ is a near $(r-1)$-regular graph, and $e(L)=\lfloor\frac{r-1}{2}(n-i+1)\rfloor$. Let $G(n,i,r)=K_{i-1}+L$. Now, let $G(n,F)$ be any graph $G(n,i,r_i)$ where $e(G(n,i,r_i))$ is maximised over $1\le i\le k$. Note that each of $G(n,i,r_i)$ and $G(n,F)$ can be one of many possible graphs.

Observe that $G(n,i,r_i)$ is $F$-free for all $1\le i\le k$. Indeed, if $G(n,i,r_i)=K_{i-1}+L$ as defined and contains a copy of $F$, then each star $S_{r_1},\dots,S_{r_{i-1}}$ must have at least one vertex from the $K_{i-1}$, and $S_{r_i}$ is not a subgraph of $L$.

Lidick\'y et al.~\cite{LLP13} proved that the graphs $G(n,F)$ are extremal for $F$.

\begin{thm}\label{LLPthm2}\,\textup{\cite{LLP13}}
Let $k\ge 2$, and $F=\bigcup_{i=1}^kS_{r_i}$ be a star forest, where $r_{1}\geq \cdots \geq r_{k}\ge 1$ are the maximum degrees of the components. Let $n\ge n_0(F)$ be sufficiently large. We have
\begin{align*}
\textup{ex}(n,F) &= e(G(n,F))\\
&= \max_{1\leq i \leq k}\Big\{(i-1)(n-i+1)+\binom{i-1}{2}+\Big\lfloor\frac{r_{i}-1}{2}(n-i+1)\Big\rfloor\Big\}.
\end{align*}
Moreover, the extremal graphs are the graphs $G(n,F)$.
\end{thm}

\section{Linear and star forests}\label{linstarforestsect}
We now study the function $\textup{ex}_p(n,F)$, where $F$ is a linear forest or a star forest, $p\ge 2$, and $n$ is sufficiently large. We assume throughout this section that $F$ has at least two components, since the single component case is covered by (\ref{CYobs1}), Theorem \ref{CYthm1}, and Proposition \ref{CYprp1}.
 
We first consider the case when $F$ is a star forest. Recall that $S_r$ is the star with maximum degree $r$. Let $F=\bigcup_{i=1}^kS_{r_i}$, where $r_{1}\geq \cdots \geq r_{k}\ge 1$. Our first result is the following. It turns out that $\textup{ex}_p(n,F)$ is attained by the graphs $G(n,k,r_k)$.

\begin{thm}\label{SFthm}
Let $k,p\ge 2$, and $F=\bigcup_{i=1}^kS_{r_i}$ be a star forest, where $r_{1}\geq \cdots \geq r_{k}\ge 1$ are the maximum degrees of the components. Let $n\ge n_0(F)$ be sufficiently large. We have
\[
\textup{ex}_{p}(n,F)=e_{p}(G(n,k,r_k)).
\]
Moreover, the extremal graphs are the graphs $G(n,k,r_k)$.
\end{thm}

%We recall, as mentioned in Section \ref{linearforestsect}, that we may obtain Theorem \ref{kP3thm} from Theorem \ref{SFthm}. That is, the result for $\textup{ex}_p(n,kP_3)$. To see this, it suffices to set $r_1=\cdots=r_k=2$ in Theorem \ref{SFthm}, and check that we have $K_{k-1}+M_{n-k+1}$ as the unique extremal graph. Indeed, we have $G(n,k,r_k)=G(n,k,2)=K_{k-1}+L$, where $L$ is a graph on $n-k+1$ vertices which is extremal for $P_3$. Thus, we have $L=M_{n-k+1}$, as required.
%
%\begin{proof}[Proof of Theorem \ref{SFthm}.]
\begin{proof}
Since $G(n,k,r_k)$ does not contain a copy of $F$, we have $\textup{ex}_{p}(n,F)\geq e_{p}(G(n,k,r_k))$. To prove the theorem, it suffices to show that any $F$-free graph $G$ on $n$ vertices with $G\neq G(n,k,r_k)$ satisfies $e_{p}(G)<e_{p}(G(n,k,r_k))$.

It is easy to calculate that
\begin{align}
e_{p}(G(n,k,r_k)) &= \left\{ 
\begin{array}{l}
(k-1)(n-1)^{p}+(n-k+1)(r_{k}+k-2)^{p}\\
\hspace{3cm}\textup{if one of $r_{k}-1$ and $n-k+1$ is even}\\[1ex]
(k-1)(n-1)^{p}+(n-k)(r_{k}+k-2)^{p}+(r_{k}+k-3)^{p}\\
\hspace{3cm}\textup{if $r_{k}-1$ and $n-k+1$ are odd}
\end{array} \right.\nonumber\\[1ex]
&= (k-1)n^{p}+o(n^{p}).\label{epHnSasy}
\end{align}

We may assume that there exists a subset $U\subset V(G)$ of order $k-1$ such that every vertex $v\in U$ has degree $d_G(v)\ge \sum_{i=1}^kr_i+k$. Otherwise, if $G$ has at most $k-2$ such vertices, then using $p\ge 2$ and (\ref{epHnSasy}), we have
\begin{align*}
e_{p}(G) &< (k-2)n^{p}+(n-k+2)\bigg(\sum_{i=1}^kr_i+k\bigg)^p<(k-1)n^{p}+o(n^{p})\\
&= e_{p}(G(n,k,r_k)).
\end{align*}

Now, we prove that $G \subset G(n,k,r_k)$, which implies that $e_{p}(G)<e_{p}(G(n,k,r_k))$. Recall that $G(n,k,r_k)=K_{k-1}+L$, where $L$ is a graph on $n-k+1$ vertices which is an extremal graph for $S_{r_k}$. Thus by identifying $U$ with $K_{k-1}$ and $G-U$ with $L$, we have $G \subset G(n,k,r_k)$ if we can show that $G-U$ is $S_{r_k}$-free. Suppose that there is a copy of $S_{r_k}$ in $G-U$. Then, using the fact that $d_G(v)\ge \sum_{i=1}^kr_i+k$ for every $v\in U$, we can find vertex-disjoint copies of $S_{r_1},\dots,S_{r_{k-1}}$, using vertices in $U$ as their centres, and with their neighbours in $G-U$, not using the vertices of the $S_{r_k}$, as leaves. This gives a copy of $F$ in $G$, a contradiction.
\end{proof}

%By setting $r_1\le 2$ in Theorem \ref{SFthm}, we can easily obtain the result for $\textup{ex}_p(n,F)$, where $F$ is a linear forest whose components have orders at most $3$. This result can be considered an extension to Theorems \ref{EGthm2} and \ref{YZthm} for large $n$.
%
%\begin{cor}\label{SFcor}
%Let $k,p\geq 2$, and $F=\bigcup_{i=1}^kP_{\ell_i}$ be a linear forest, where $3\ge \ell_{1}\ge \ell_{2}\ge\cdots\ge \ell_{k}\ge 2$. Let $n\ge n_0(F)$ be sufficiently large.
%\begin{enumerate}
%\item[(a)] If $F\neq kP_3$, then $\textup{ex}_p(n,F)=e_p(K_{k-1}+E_{n-k+1})$. In particular, we have $\textup{ex}_p(n,kP_2)=e_p(K_{k-1}+E_{n-k+1})$. Moreover, $K_{k-1}+E_{n-k+1}$ is the unique extremal graph.
%\item[(b)] If $F=kP_3$, then $\textup{ex}_p(n,F)=e_p(K_{k-1}+M_{n-k+1})$. Moreover, $K_{k-1}+M_{n-k+1}$ is the unique extremal graph.
%\end{enumerate}
%\end{cor}

Now, we consider the case when $F$ is an arbitrary linear forest. Let $F=\bigcup_{i=1}^kP_{\ell_i}$, where $k\ge 2$ and $\ell_1\ge\cdots\ge\ell_k\ge 2$. For the case when $F=kP_3$, we can set $r_1=\cdots=r_k=2$ in Theorem \ref{SFthm} to obtain the following result, which can be considered as an extension to Theorem \ref{YZthm}.

\begin{cor}\label{SFcor}
Let $k,p\ge 2$ and $n\ge n_0(k)$ be sufficiently large. We have $\ex_p(n,kP_3)=e_p(K_{k-1}+M_{n-k+1})$. Moreover, $K_{k-1}+M_{n-k+1}$ is the unique extremal graph.
\end{cor}

Now, let $F\neq kP_3$. We shall prove the following result, which can be considered as an extension to Theorem \ref{LLPthm1}.

\begin{thm}\label{LFthm}
Let $k,p\ge 2$, and $F=\bigcup_{i=1}^kP_{\ell_i}$ be a linear forest, where $\ell_{1}\ge \ell_{2}\ge\cdots\ge \ell_{k}\ge 2$ and $\ell_i\neq 3$ for some $i$. Let $n\ge n_0(F)$ be sufficiently large. We have
\[
\textup{ex}_{p}(n,F)=e_{p}(H(n,F)).
\]
Moreover, $H(n,F)$ is the unique extremal graph.

In particular, if $F=kP_\ell$ and $\ell\neq 3$, then 
\[
\textup{ex}_{p}(n,kP_{\ell})=e_{p}(H(n,k,\ell)).
\]
Moreover, $H(n,k,\ell)$ is the unique extremal graph.
\end{thm}

Before we prove Theorem \ref{LFthm}, we first recall a lemma of Caro and Yuster \cite{CY00}.

%\begin{lem}\label{CYlem1}\,\textup{\cite{CY00}}
%Let $p\geq 2$ be an integer, let $0.5< \alpha \leq 1$ and let $t >\alpha$ be real. Let $G$ be a graph of order $n$ with at most $tn$ edges and maximum degree at most $\alpha n$. Then
%\[
%e_{p}(G)\leq \frac{t}{\alpha}(\alpha n)^{p}+o(n^{p}).
%\]
%\end{lem}

\begin{lem}\label{CYlem2}\,\textup{\cite{CY00}}
Let $b\ge 1$ and $p\ge 2$ be integers. Let $G$ be a graph on $n$ vertices such that $e(G)\le (b+\frac{1}{2})n$. Let $d_1\ge \cdots\ge d_n$ be the degree sequence of $G$. Then, if $d_b\le 0.65n$, we have $e_p(G)\le cn^p+O(n^{p-1})$ for some constant $c$ with $0<c<b$.
\end{lem}

Although Lemma \ref{CYlem2} is not stated explicitly in \cite{CY00}, it can be seen easily in the proof of Lemma 3.5 in \cite{CY00}.

\begin{proof}[Proof of Theorem \ref{LFthm}] 
Since $H(n,F)$ is $F$-free, we have $\textup{ex}_p(n,F)\ge e_p(H(n,F))$. Hence, it suffices to show that any $F$-free graph $G$ on $n$ vertices with $G\neq H(n,F)$ has $e_{p}(G)<e_{p}(H(n,F))$. Assume the contrary, and let $G$ be an $F$-free graph on $n$ vertices, that is maximal in the sense that $e_{p}(G)=\textup{ex}_{p}(n,F)\ge e_p(H(n,F))$ and $G \neq H(n,F)$.

Let $b=\sum_{i=1}^{k}\lfloor\frac{\ell_{i}}{2}\rfloor-1\ge 1$. By the definition of $H(n,F)$, it is easy to calculate that
\begin{align}
e_{p}(H(n,F)) &= \left\{ 
\begin{array}{ll}
b(n-1)^{p}+(n-b-2)b^p+2(b+1)^{p} & \textup{if all $\ell_i$ are odd}\\
b(n-1)^{p}+(n-b)b^p & \textup{otherwise}
\end{array} 
\right.\nonumber\\[1ex]
&= bn^{p}+o(n^{p}).\label{ep(HnF)eq}
\end{align}

According to Theorem \ref{LLPthm1}, we have
\begin{equation}\label{kPleq}
e(G)\leq {b\choose 2}+b(n-b)+1\le bn.
\end{equation}

Let $d_1\ge\cdots\ge d_n$ be the degree sequence of $G$. Let $X\subset V(G)$ be the set of vertices with degrees $d_1,\dots,d_b$, and $Y=V(G)\setminus X$. By Lemma \ref{CYlem2} and using (\ref{ep(HnF)eq}) and (\ref{kPleq}), we may assume that $d_{b}>0.65n$. Let $A\subset Y$ be the set of vertices that have a neighbour in $X$, and $B=Y\setminus A$. Note that any two vertices $u,v\in X$ have at least $2(0.65n-1)-(n-2)=0.3n$ common neighbours in $G$, and hence at least $0.3n-(b-2)>0.29n$ common neighbours in $A$. This means that for any set $Y'\subset Y$ with $|Y'|$ depending only on $\ell_1,\dots,\ell_k$ (and hence $|Y'|$ is much smaller than $n$), $u$ and $v$ have a common neighbour in $A\setminus Y'$. Likewise, any vertex of $X$ has a neighbour in $A\setminus Y'$. We now prove a series of claims.

\begin{claim}\label{clm1}
If $\ell_i$ is odd for all $i$, then $G[Y]$ does not contain a copy of $P_3$ with an end-vertex in $A$. If $\ell_i$ is even for some $i$, then $G[Y]$ does not contain an edge with an end-vertex in $A$.
\end{claim}

\begin{proof}
Suppose first that $\ell_i$ is odd for all $i$, and that $G[Y]$ contains a path $c_1c_2a_1$ with $a_1\in A$. Let $y_1\in X$ be a neighbour of $a_1$, and $y_2\in X\setminus\{y_1\}$. Then, $y_1$ and $y_2$ have a common neighbour $a_2\in A\setminus\{c_1,c_2,a_1\}$. Repeating this procedure, we can obtain a path $c_1c_{2}a_{1}y_{1}a_{2}y_{2}\ldots y_{b-1}a_{b}y_{b}a_{b+1}$, where $X=\{y_1,\dots,y_b\}$ and $a_2,a_3,\dots,a_{b+1}\in A$. This path has $2b+3=\ell_1+\sum_{i=2}^k(\ell_i-1)$ vertices, and so it contains vertex-disjoint paths $P_{\ell_1}, P_{\ell_2-1},\dots, P_{\ell_k-1}$ with $c_1$ in the $P_{\ell_1}$. Note that each of the paths $P_{\ell_2-1},\dots, P_{\ell_k-1}$ has an end-vertex in $X$, and so we can extend each $P_{\ell_i-1}$ to $P_{\ell_i}$ by taking a neighbour of the end-vertex in $X$. By choosing the $k-1$ neighbours to be distinct vertices in $A\setminus\{c_1,c_2,a_1,\dots,a_{b+1}\}$, we obtain a copy of $F$ in $G$, a contradiction.

Now, let $Q=\{1\le i\le k:\ell_i$ is even$\}\neq\emptyset$, and suppose that $G[Y]$ contains an edge $ca_1$ with $a_1\in A$. As before, we can obtain a path $ca_1y_1a_2y_2\dots y_{b-1}a_by_ba_{b+1}$, where $X=\{y_1,\dots,y_b\}$ and $a_2,a_3,\dots,a_{b+1}\in A\setminus\{c,a_1\}$. This path has $2b+2=\sum_{i\not\in Q}(\ell_i-1)+\sum_{i\in Q}\ell_i$ vertices. We obtain vertex-disjoint paths $P_{\ell_i-1}$ for $i\not\in Q$, and $P_{\ell_i}$ for $i\in Q$, such that the path using $c$ is $P_{\ell_j}$, for some $j\in Q$. Extending each $P_{\ell_i-1}$ to $P_{\ell_i}$ for $i\not\in Q$, we again have a copy of $F$ in $G$, a contradiction.
\end{proof}

\begin{claim}\label{clm2}
$G[B]$ does not contain a copy of $P_{\ell_k}$.
\end{claim}

\begin{proof}
Suppose that $G[B]$ contains a copy of $P_{\ell_k}$. Let $Q_0=\{1\le i\le k-1:\ell_i$ is even$\}$ and $Q_1=\{1\le i\le k-1:\ell_i$ is odd$\}$. We can find a path $a_{1}y_{1}a_{2}y_{2}\ldots a_{b}y_{b}$, where $X=\{y_1,\dots,y_b\}$ and $a_1,\dots,a_b\in A$. This path has $2b\ge\sum_{i\in Q_0}\ell_i+\sum_{i\in Q_1}(\ell_i-1)$ vertices, and hence contains vertex-disjoint paths $P_{\ell_i}$ for $i\in Q_0$, and $P_{\ell_i-1}$ for $i\in Q_1$. Extending each $P_{\ell_i-1}$ to $P_{\ell_i}$ for $i\in Q_1$, we have a copy of $F$ in $G$, a contradiction.
\end{proof}

\begin{claim}\label{clm3} 
$B=\emptyset$.
\end{claim}

\begin{proof}
Suppose that $B\neq\emptyset$. If $\ell_k=2$, then note that Claims \ref{clm1} and \ref{clm2} imply that $G[Y]$ does not contain any edges. This means that $G$ is a subgraph of $H(n,F)$, and we have $e_p(G)<e_p(H(n,F))$, a contradiction. Now, let $\ell_k\ge 3$. We will derive a contradiction by constructing a new $F$-free graph $G'$ such that $e_{p}(G')>e_{p}(G)$.

Note that by Claim \ref{clm1}, every vertex of $B$ has at most one neighbour in $G$ lying in $A$. Since Claim \ref{clm2} implies that $G[B]$ is $P_{\ell_k}$-free, by Theorem \ref{EGthm1}, $G[B]$ contains at most $(\frac{\ell_k}{2}-1)|B|$ edges. Hence, there exists a vertex $v\in B$ with at most $\ell_k - 2$ neighbours in $G[B]$. Now in $G$, in view of Claim \ref{clm1}, one of the following holds.
\begin{enumerate}
\item[(i)] $d_G(v)=1$, with the only neighbour of $v$, say $u$, lying in $A$.
\item[(ii)] $0\le d_G(v)\le \ell_k-2$, with all neighbours of $v$ lying in $B$.
\end{enumerate}

Delete all edges adjacent to $v$ in $G$, connect $v$ to all vertices of $X$, and denote the new graph by $G'$. We claim that $G'$ is also $F$-free. Indeed, if $G'$ contains a copy of $F$, then exactly one path of $F$, say $P_{\ell_j}$, must use an edge $vy_1$, for some $y_1\in X$. If $v$ is not an end-vertex of such a $P_{\ell_j}$, then the $P_{\ell_j}$ also contains another neighbour $y_2\in X$ of $v$. We can find a common neighbour $v'\in A$ of $y_1$ and $y_2$ in $G$ which is not used in the copy of $F$. If $v$ is an end-vertex of the $P_{\ell_j}$, then we take $v'\in A$ to be any neighbour of $y_1$ not in the copy of $F$. Replacing $v$ with $v'$ on the $P_{\ell_j}$, we obtain a copy of $F$ in $G$, a contradiction. 

We now show that $e_p(G') > e_p(G)$. Consider the effect of the transformation from $G$ to $G'$ on the degree sequence. The degrees of the vertices of $X$ have increased by one. The degree of $v$ has not decreased, since $d_{G'}(v)-d_G(v)\ge b-(\ell_k -2)\ge (2\lfloor\frac{\ell_k}{2}\rfloor-1)-(\ell_k-2)\ge 0$. The degrees of the neighbours of $v$ in $G$ have decreased by $1$. Since every vertex of $X$ has degree at least $0.65n$, the total increase in $e_p(G') - e_p(G)$ contributed by the vertices of $X$ is at least
\[
b((0.65n + 1)^p - (0.65n)^p) = bp(0.65n)^{p-1} + o(n^{p-1}).
\]

Also, $d_b > 0.65n$ implies that $|B| < 0.35n$. If (i) holds, then Claim \ref{clm1} implies that in $G$, $u$ has no neighbours in $A$, and hence $d_G(u)<0.35n+b<0.36n$. The decrease in $e_p(G') - e_p(G)$ contributed by $u$ is at most
\[
(0.36n)^p-(0.36n-1)^p=p(0.36n)^{p-1}+o(n^{p-1}).
\]

Suppose that (ii) holds. Then in $G$, every neighbour of $v$ (lying in $B$) cannot have a neighbour in $A$, in view of Claim \ref{clm1}. Hence, every neighbour of $v$ has degree at most $0.35n$ in $G$. The total decrease in $e_p(G') - e_p(G)$ contributed by the neighbours of $v$ is at most
\[
(\ell_k -2)((0.35n)^p - (0.35n - 1)^p) \le (\ell_k - 2)p(0.36n)^{p-1} + o(n^{p-1}).
\]

Hence,
\[
e_p(G') - e_p(G)\ge p(b(1.8)^{p-1} - \ell_k + 2)(0.36)^{p-1}n^{p-1} + o(n^{p-1}) > 0.
\]
\end{proof}

By Claim \ref{clm3}, we may assume that $A=Y$ for the rest of the proof.

\begin{claim}\label{clm4}  
If $\ell_i$ is odd for all $i$, then $G[Y]$ contains at most one edge. If $\ell_i$ is even for some $i$, then $G[Y]$ does not contain an edge.
\end{claim}

\begin{proof}
The latter assertion follows immediately from Claim \ref{clm1} and the fact that $A=Y$. Now, assume that $\ell_i$ is odd for all $i$. Then, since we do not have $\ell_1=\cdots=\ell_k=3$, we have $\ell_1\ge 5$. Assuming the contrary, Claim \ref{clm1} implies that the subgraph $G[Y]$ is a set of at least two independent edges and isolated vertices. We consider three cases.\\[1ex]
\noindent\emph{Case 1.} \emph{
\begin{itemize}
\item Either $\ell_1=5$, and there are two edges $c_1c_2, c_3c_4$ in $G[Y]$, with $c_2,c_3$ having a common neighbour in $X$;
\item Or $\ell_1\ge 7$, and there are two edges $c_1c_2,c_3c_4$ in $G[Y]$, with $c_2,c_3$ having distinct neighbours in $X$.
\end{itemize}}

Let $q=\lfloor\frac{\ell_1}{2}\rfloor-1$. For $\ell_1=5$, let $c_2,c_3$ have a common neighbour $y_1\in X$, so that $c_1c_2y_1c_3c_4$ is a copy of $P_5$. For $\ell_1\ge 7$, let $c_2,c_3$ have distinct neighbours $y_1,y_q\in X$. Then, as before, we can find a copy of $P_{\ell_1}$ in the form $c_1c_2y_1a_2y_2\dots a_qy_qc_3c_4$, where $y_2,\dots,y_{q-1}\in X$ and $a_2,\dots,a_q\in Y$. In both cases, we can again find a path $a_{q+1}y_{q+1}a_{q+2}y_{q+2}\dots a_by_b$, where $X=\{y_1,\dots,y_b\}$ and $a_{q+1},a_{q+2},\dots,a_{b}\in Y\setminus\{c_1,c_2,c_3,c_4,a_2,\dots,a_q\}$. This path has $2(b-q)=\sum_{i=2}^k(\ell_i-1)$ vertices, and so contains vertex-disjoint copies of $P_{\ell_2-1},\dots,P_{\ell_k-1}$. As before, we can extend these to copies of $P_{\ell_2},\dots,P_{\ell_k}$ so that we have a copy of $F$ in $G$, a contradiction.\\[1ex]
\noindent\emph{Case 2.} \emph{$\ell_1=5$, and no two vertices from distinct edges in $G[Y]$ have a common neighbour in $X$.}\\[1ex]
%\indent Note that $G$ is $F$-free. Indeed, if $G$ contains a copy of $F$, then the definition of $b$ implies that either some copy of $P_5$ uses at least four vertices of $Y$, or some copy of $P_3$ is contained in $G[Y]$. But clearly, neither of these is possible. 
\indent We shall prove that $e_p(G)<e_p(H(n,F))$, which will contradict the choice of $G$. Recall that $H(n,F)$ is $K_b+E_{n-b}$ with an edge $uv$ added to the empty class, and note that $b\ge\lfloor\frac{5}{2}\rfloor+\lfloor\frac{3}{2}\rfloor-1=2$. Let $d_H(z)$ denote the degree of a vertex $z$ in $H(n,F)$. Clearly, for every $z\in X$ and every vertex $z'$ in the $K_b$, we have $d_G(z)\le n-1=d_H(z')$. Now, let $u_1u_2,\dots,u_{2s-1}u_{2s}$ be all the independent edges in $G[Y]$, for some $s\ge 2$, and let $\Gamma_i$ be the set of vertices in $X$ that are adjacent to at least one of $u_{2i-1}$ and $u_{2i}$, for $i=1,\dots,s$. We may assume that $|\Gamma_1|\ge\cdots\ge|\Gamma_s|\ge 1$. Note that the $\Gamma_i$ are pairwise disjoint subsets of $X$, so that $|\Gamma_2|\le \frac{b}{2}$. Also, since $u_3$ has a neighbour in $X$, we have $|\Gamma_1|\le b-1$. Hence, $d_G(u_j)\le b<d_H(u)=d_H(v)$ for $j=1,2$; $d_G(u_j)\le \frac{b}{2}+1\le b=d_H(z')$ for $j=3,\dots,2s$; and $d_G(z)\le b=d_H(z')$ for $z\in Y\setminus\{u_1,\dots,u_{2s}\}$ and $z'\neq u,v$ in the $E_{n-b}$. The degree sequence of $H(n,F)$ strictly majorises that of $G$, and therefore, we have $e_p(G)<e_p(H(n,F))$.\\[1ex]
\noindent\emph{Case 3.} \emph{$\ell_1\ge 7$, and all vertices of the edges of $G[Y]$ are connected to a single vertex of $X$.}\\[1ex]
\indent Let $y_1\in X$ be this single vertex, and note that $\ell_1\ge 7$ implies that $b\ge 3$. We construct an $F$-free graph $G'$ on $n$ vertices such that $e_{p}(G')>e_{p}(G)$, which contradicts the choice of $G$. Let $y_{2} \in X\setminus\{y_1\}$, and let $Y^*$ denote the set of non-isolated vertices in $G[Y]$. Observe that $|Y^*|\geq 4$ and no vertex of $Y^*$ is adjacent to any vertex of $X\setminus\{y_1\}$. We construct $G'$ as follows: delete the $\frac{|Y^*|}{2}$ independent edges of $G[Y]$, and join $y_{2}$ to each vertex of $Y^*$. 

Similar to Claim \ref{clm3}, we see that $G'$ is an $F$-free graph. Indeed, if $G'$ contains a copy of $F$, then exactly one path, say $P_{\ell_j}$, must use an edge $vy_2$, for some $v\in Y^*$. If $v$ is not an end-vertex of such a $P_{\ell_j}$, then the other neighbour of $v$ in the $P_{\ell_j}$ is $y_1$. Now in $G$, we can find a common neighbour $v'\in Y$ of $y_1$ and $y_2$ which is not used in the copy of $F$. If $v$ is an end-vertex of the $P_{\ell_j}$, then we can take $v'\in Y$ to be any neighbour of $y_2$ not in the copy of $F$. Replacing $v$ with $v'$ on the $P_{\ell_j}$, we obtain a copy of $F$ in $G$, a contradiction.

However, the degree sequence of $G'$ strictly majorises that of $G$, since the degree of $y_2$ has strictly increased, and all other degrees have not changed. Hence $e_p(G')>e_p(G)$, which is the required contradiction.
\end{proof}

By Claim \ref{clm4}, $G$ is a spanning subgraph of $H(n,F)$. Hence $e_{p}(G)<e_{p}(H(n,F))$, which contradicts the choice of $G$. The proof of Theorem \ref{LFthm} is complete.
\end{proof}

\section{Brooms}\label{broomsect}

In this section, we shall consider the function $\ex_p(n,B_{\ell,s})$, where $B_{\ell,s}$ is a broom graph, $p\ge 2$, $\ell\ge 4$, $s\ge 0$, and $n$ is sufficiently large. As we have already seen, Theorems \ref{SWthm1} and \ref{SWthm2} appear to suggest that the determination of the Tur\'an function $\ex(n,B_{\ell,s})$ and the corresponding extremal graphs may be a complicated problem, in the sense that the potential results may be difficult to state. Somewhat surprisingly, we shall see here that the same problem for $\ex_p(n,B_{\ell,s})$, where $p\ge 2$, may possibly be more manageable. Since the case $\ell=4$ is covered in Proposition \ref{CYprp2}, we consider $\ell\ge 5$. Here, we will provide the answers for the cases $\ell=5,6,7$, and present a conjecture for the case of general $\ell$. The case $\ell=5$ turns out to be a rather special case. Although the case $s=0$ is covered by Theorem \ref{CYthm1}, we will include this case here since we will obtain some explicit lower bounds for $n$.

\begin{thm}\label{B5sthm}
Let $p\ge 2$, $s\ge 0$, and $n>(2s+10)^2$. We have
\[
\ex_p(n,B_{5,s})=
\left\{
\begin{array}{ll}
e_p(H(n,5)) & \textup{\emph{if $s=0$,}}\\
e_p(K_1+M_{n-1}) & \textup{\emph{if $s\ge 1$.}}
\end{array}
\right.
\]
Moreover, the unique extremal graph is $H(n,5)$ if $s=0$, and $K_1+M_{n-1}$ if $s\ge 1$.
\end{thm}

\begin{thm}\label{B6sthm}
Let $p\ge 2$, $s\ge 0$, and $n>(2s+12)^2$. We have
\[
\ex_p(n,B_{6,s})=e_p(H(n,6)).
\]
Moreover, $H(n,6)$ is the unique extremal graph.
\end{thm}

\begin{thm}\label{B7sthm}
Let $p\ge 2$, $s\ge 0$, and $n>(3s+31)^2$. We have
\[
\ex_p(n,B_{7,s})=e_p(H(n,7)).
\]
Moreover, $H(n,7)$ is the unique extremal graph.
\end{thm}

In view of Theorems \ref{B6sthm} and \ref{B7sthm}, we believe that the following assertion may be true.

\begin{conj}\label{Blsconj}
Let $p\ge 2$, $\ell\ge 6$, $s\ge 0$, and $n\ge n_0(\ell,s)$ be sufficiently large. We have
\[
\ex_p(n,B_{\ell,s})=e_p(H(n,\ell)).
\]
Moreover, $H(n,\ell)$ is the unique extremal graph.
\end{conj}

That is, Conjecture \ref{Blsconj} claims that if $n$ is sufficiently large, then $\ex_p(n,B_{\ell,s})$ is exactly the same as $\ex_p(n,P_{\ell})$, with the same unique extremal graph $H(n,\ell)$. If Conjecture \ref{Blsconj} is true, then it can be considered as an extension to Theorem \ref{CYthm1}.

Before we prove Theorems \ref{B5sthm} to \ref{B7sthm}, we first prove some auxiliary lemmas. We also prove a proposition which simplifies a possible proof of Conjecture \ref{Blsconj}.

\begin{lemma}\label{Blslem1}
Let $p\ge 2$, $n_1,n_2\ge\ell$, and $n=n_1+n_2$.
\begin{enumerate}
\item[(a)] If $\ell=5$, then $e_p(K_1+M_{n_1-1})+e_p(K_1+M_{n_2-1})<e_p(K_1+M_{n-1})$.
\item[(b)] If $\ell\ge 5$, then $e_p(H(n_1,\ell))+e_p(H(n_2,\ell))<e_p(H(n,\ell))$.
\end{enumerate}
\end{lemma}

\begin{proof}
(a) Let $\ell=5$. Then
\begin{align*}
e_p(K_1+M_{n_1-1})+e_p(K_1+M_{n_2-1}) &\le (n_1-1)^p+(n_1-1)2^p\\
&\quad\quad\quad +(n_2-1)^p+(n_2-1)2^p\\
&< (n-1)^p+(n-2)2^p< e_p(K_1+M_{n-1}).
\end{align*}
\indent (b) Let $\ell\ge 5$, and $b=\lfloor\frac{\ell}{2}\rfloor-1\ge 1$. Since
\begin{align*}
e_p(H(n_1,\ell))+e_p(H(n_2,\ell)) &\le b(n_1-1)^p+(n_1-b-2)b^p+2(b+1)^p\\
&\quad\quad\quad +b(n_2-1)^p+(n_2-b-2)b^p+2(b+1)^p,\\
e_p(H(n,\ell)) &\ge b(n-1)^p+(n-b)b^p,
\end{align*}
it suffices to prove that
\begin{equation}
b(n-1)^p>b[(n_1-1)^p+(n_2-1)^p]+4(b+1)^p-(b+4)b^p.\label{Blslemeq1}
\end{equation}
Clearly, $n\ge 2\ell\ge 4b+4$. We have
\begin{align*}
b(n-1)^p &>b(n-2)^p+bp(n-2)^{p-1}\\
&>b[(n_1-1)^p+(n_2-1)^p]+2b(4b+2)^{p-1},
\end{align*}
which implies (\ref{Blslemeq1}), since it is easy to verify that $2b(4b+2)^{p-1}> 4(b+1)^p-(b+4)b^p$.
\end{proof}

\begin{lemma}\label{Blslem2}
%Let $p\ge 2$, $s\ge 0$, $\ell\ge 5$, $n>(\ell+s)^2$, and $n>h\ge \ell+s$. Let $G^\ast$ be a graph on $n-h$ vertices with $\Delta(G^\ast)\le \ell+s-2$.
%\begin{enumerate}
%\item[(a)] If $\ell=5$, then $e_p(K_1+M_{h-1})+e_p(G^\ast)<e_p(K_1+M_{n-1})$.
%\item[(b)] If $\ell\ge 5$, then $e_p(H(h,\ell))+e_p(G^\ast)<e_p(H(n,\ell))$.
%\end{enumerate}
%
Let $p\ge 2$, $s\ge 0$, $\ell\ge 5$. Let $G^\ast$ be a graph on $h^\ast>0$ vertices with $\Delta(G^\ast)\le d=d(\ell,s)$. Let $n=h+h^\ast>(\ell+s+d)^2$ for some $h\ge \ell$.
\begin{enumerate}
\item[(a)] If $\ell=5$, then $e_p(K_1+M_{h-1})+e_p(G^\ast)<e_p(K_1+M_{n-1})$.
\item[(b)] If $\ell\ge 5$, then $e_p(H(h,\ell))+e_p(G^\ast)<e_p(H(n,\ell))$.
\end{enumerate}
\end{lemma}

\begin{proof}
%(a) Let $\ell=5$. Since
%\begin{align*}
%e_p(K_1+M_{h-1})+e_p(G^\ast) &\le (h-1)^p+(h-1)2^p+(n-h)(s+3)^p,\\
%e_p(K_1+M_{n-1}) &> (n-1)^p+(n-2)2^p,
%\end{align*}
%and $(n-2)2^p\ge (h-1)2^p$, it suffices to prove that
%\[
%(n-1)^p-(h-1)^p>(n-h)(s+3)^p.
%\]
%We have
%\[
%(n-1)^p-(h-1)^p>(n-h)(n-1)^{p-1}\ge (n-h)(5+s)^{2p-2}>(n-h)(s+3)^p,
%\]
%as required.\\[1ex]
%\indent (b) Let $\ell\ge 5$, and $b=\lfloor\frac{\ell}{2}\rfloor-1\ge 1$. Since
%\begin{align*}
%e_p(H(h,\ell))+e_p(G^\ast) &\le b(h-1)^p+(h-b-2)b^p+2(b+1)^p+(n-h)(\ell+s-2)^p,\\
%e_p(H(n,\ell)) &\ge b(n-1)^p+(n-b)b^p,
%\end{align*}
%and $(n-b)b^p>(h-b-2)b^p$, it suffices to prove that
%\begin{equation}
%b[(n-1)^p-(h-1)^p]>(n-h)(\ell+s-2)^p+2(b+1)^p.\label{Blslemeq2}
%\end{equation}
%
%If $\ell=5$, then $b=1$. We have
%\begin{align*}
%(n-1)^p-(h-1)^p &>(n-h)(n-1)^{p-1}\ge (n-h)(5+s)^{2p-2}\ge (n-h)(5+s)^p\\
%&>(n-h)(s+3)^p+2p(n-h)(s+3)^{p-1}\\
%&\ge (n-h)(s+3)^p+4\cdot 3^{p-1}\\
%&>(n-h)(s+3)^p+2^{p+1}.
%\end{align*}
%
%If $\ell\ge 6$, then $b\ge 2$. Clearly $\ell+s\ge\ell\ge 2b+2$. We have
%\begin{align*}
%b[(n-1)^p-(h-1)^p] &>2(n-h)(n-1)^{p-1}\ge 2(n-h)(\ell+s)^{2p-2}\\
%&\ge (n-h)(\ell+s)^p+(2b+2)^p\\
%&>(n-h)(\ell+s-2)^p+2(b+1)^p.
%\end{align*}
%
%Therefore (\ref{Blslemeq2}) holds, as required.
%
(a) Let $\ell=5$. We have
\begin{align*}
e_p(K_1+M_{h-1})+e_p(G^\ast) &\le (h-1)^p+(h-1)2^p+h^\ast d^p,\\
e_p(K_1+M_{n-1}) &> (n-1)^p+(n-2)2^p.
\end{align*}

Since $(n-2)2^p\ge (h-1)2^p$, and
\[
(n-1)^p-(h-1)^p>(n-h)(n-1)^{p-1}>h^\ast d^{2p-2}\ge h^\ast d^p,
\]
it follows that $e_p(K_1+M_{h-1})+e_p(G^\ast)<e_p(K_1+M_{n-1})$.\\[1ex]
\indent (b) Let $\ell\ge 5$, and $b=\lfloor\frac{\ell}{2}\rfloor-1\ge 1$. Since
\begin{align*}
e_p(H(h,\ell))+e_p(G^\ast) &\le b(h-1)^p+(h-b-2)b^p+2(b+1)^p+h^\ast d^p,\\
e_p(H(n,\ell)) &\ge b(n-1)^p+(n-b)b^p,
\end{align*}
and $(n-b)b^p>(h-b-2)b^p$, it suffices to prove that
\[
b[(n-1)^p-(h-1)^p]>2(b+1)^p+h^\ast d^p.
\]
Clearly $\ell+s\ge\ell\ge 2b+2$. We have
\begin{align*}
b[(n-1)^p-(h-1)^p] &>(n-h)(n-1)^{p-1}\ge h^\ast(\ell+s+d)^{2p-2}\\
&\ge h^\ast(2b+2+d)^p> h^\ast(2b+2)^p+h^\ast d^p\\
&>2(b+1)^p+h^\ast d^p,
\end{align*}
as required.
\end{proof}

Before we prove the next lemma, we make some definitions. Let $C$ be a connected graph, and $v,x\in V(C)$.
\begin{itemize}
\item For $y\in V(C-\{v,x\})$, the edge $e=xy\in E(C)$ is an \emph{$x$-pendent edge} if $x$ is the only neighbour of $y$ in $C$.
\item Let $y,y'\in V(C-\{v,x\})$ where $xy,xy',yy'\in E(C)$, and $y,y'$ do not have any other neighbours in $C$. The subgraph $T=C[\{x,y,y'\}]$ is an \emph{$x$-pendent triangle}.
\item Let $z,y,y'\in V(C-\{v,x\})$ where $xy,xy',zy,zy',yy'\in E(C)$, and $z,y,y'$ do not have any other neighbours in $C$. The subgraph $D=C[\{x,z,y,y'\}]$ is an \emph{$x$-pendent diamond}.
\item For some $t\ge 2$, let $z,y_1,\dots,y_t\in V(C-\{v,x\})$ where $xy_k,zy_k\in E(C)$ (resp.~$xy_k,zy_k,$ $xz\in E(C)$) for every $1\le k\le t$, and $z,y_1,\dots,y_t$ do not have any other neighbours in $C$. The subgraph $S=C[\{x,z,y_1,\dots, y_t\}]$ (resp.~$S^+=C[\{x,z,y_1,\dots, y_t\}]$) is an \emph{$x$-pendent spindle} (resp.~\emph{$x$-pendent spindle$^+$}).
\end{itemize}

\begin{lemma}\label{Blslem3}
Let $p\ge 2$, $s\ge 0$, and $\ell\ge 5$. Let $C$ be a connected $B_{\ell,s}$-free graph, and $v\in V(C)$ where $d_C(v)=\Delta(C)\ge \ell+s-1$. Let $C'$ be a graph that can be obtained from $C$ with any of the following operations.
\begin{enumerate}
\item[(i)] Delete an $x$-pendent edge $e=xy$, and add the edge $vy$.
\item[(ii)] Delete the three edges of an $x$-pendent triangle $T=C[\{x,y,y'\}]$, and add the edges $vy,vy'$.
\item[(iii)] Delete the five edges of an $x$-pendent diamond $D=C[\{x,z,y,y'\}]$, and add the edges $vz,vy,vy'$.
\item[(iv)] Delete the $2t$ edges of an $x$-pendent spindle $S=C[\{x,z,y_1,\dots, y_t\}]$ (for some $t\ge 2$), and add the edges $vz,vy_1,\dots.vy_t$.
\item[(v)] Delete the $2t+1$ edges of an $x$-pendent spindle$^+$ $S^+=C[\{x,z,y_1,\dots, y_t\}]$ (for some $t\ge 2$), and add the edges $vz,vy_1,\dots.vy_t$.
\end{enumerate}
Then $C'$ is also $B_{\ell,s}$-free, $d_{C'}(v)=\Delta(C')\ge \ell+s-1$, and $e_p(C)<e_p(C')$.
\end{lemma}

\begin{proof}
Clearly we have $d_{C'}(v)=\Delta(C')\ge \ell+s-1$, since in the transformation from $C$ to $C'$, the only vertex whose degree has increased is $v$. 

Next, let $V_1$ be the set of neighbours of $v$ in $C$. Suppose that $C'$ contains a copy of $B_{\ell,s}$, and we are in case (iv) or (v). Then for some $u_1,\dots,u_m\in \{z,y_1,\dots,y_t\}$ where $1\le m\le t+1$, the edges $vu_1,\dots,vu_m$ must be used by the $B_{\ell,s}$, with $u_1,\dots,u_m$ being leaves. Note that  $|V_1\cup\{v,y_1,\dots,y_t,z\}|\ge \ell+s+t+1$, and this means that there are vertices $w_1,\dots,w_m\in V_1$ which are not used in the $B_{\ell,s}$. Thus we obtain a copy of $B_{\ell,s}$ in $C$ by replacing $vu_1,\dots,vu_m$ with $vw_1,\dots,vw_m$, a contradiction. Similar arguments hold if we are in the other three cases, in view of $|V_1\cup\{v,y\}|\ge \ell+s+1$; $|V_1\cup\{v,y,y'\}|\ge \ell+s+2$; $|V_1\cup\{v,z,y,y'\}|\ge \ell+s+3$ for (i), (ii), (iii), respectively. Therefore, $C'$ is $B_{\ell,s}$-free.

It remains to prove that  $e_p(C)<e_p(C')$ for each case.\\[1ex]
\indent (i) Going from $C$ to $C'$, we see that the degree of $v$ is increased by $1$, and the degree of $x$ is decreased by $1$. Since $d_C(v)\ge d_C(x)\ge 2$, we have
\begin{align*}
e_p(C')-e_p(C) &= (d_{C}(v)+1)^p-d_C(v)^p+(d_{C}(x)-1)^p-d_C(x)^p\\
&\ge\sum_{1\le j\le p\textup{, $j$ odd}}{p\choose j}(d_C(v)^{p-j}-d_C(x)^{p-j})+{p\choose 2}(d_C(v)^{p-2}+d_C(x)^{p-2})\\
&>0.
\end{align*}
\indent (ii) Going from $C$ to $C'$, we see that the degree of $v$ is increased by $2$, the degree of $x$ is decreased by $2$, and the degrees of $y,y'$ are each decreased from $2$ to $1$. Since $d_C(v)\ge d_C(x)\ge 3$ and $d_C(v)\ge \ell+s-1\ge 4$, we have
\begin{align*}
e_p(C')-e_p(C) &= (d_{C}(v)+2)^p-d_C(v)^p+(d_{C}(x)-2)^p-d_C(x)^p+2(1^p-2^p)\\
&\ge\sum_{1\le j\le p\textup{, $j$ odd}}{p\choose j}(d_C(v)^{p-j}-d_C(x)^{p-j})2^j\\
&\quad\quad\quad+{p\choose 2}(d_C(v)^{p-2}+d_C(x)^{p-2})2^2+2(1-2^p)\\
&>4(4^{p-2}+3^{p-2})-2^{p+1}\ge 0.
\end{align*}
\indent(iii) Going from $C$ to $C'$, we see that the degree of $v$ is increased by $3$, the degree of $x$ is decreased by $2$, the degree of $z$ is decreased from $2$ to $1$, and the degrees of $y,y'$ are each decreased from $3$ to $1$. Since 
$d_C(v)\ge d_C(x)\ge 3$ and $d_C(v)\ge \ell+s-1\ge 4$, we have
\begin{align*}
e_p(C')-e_p(C) &= (d_{C}(v)+3)^p-d_C(v)^p+(d_{C}(x)-2)^p-d_C(x)^p\\
&\quad\quad\quad+(1^p-2^p)+2(1^p-3^p)\\
&\ge p(3d_C(v)^{p-1}-2d_C(x)^{p-1})+{p\choose 2}(9d_C(v)^{p-2}+4d_C(x)^{p-2})\\
&\quad\quad\quad+\sum_{3\le j\le p\textup{, $j$ odd}}{p\choose j}(d_C(v)^{p-j}3^j-d_C(x)^{p-j}2^j)+3-2^p-2\cdot 3^p\\
&\ge 2\cdot 4^{p-1}+9\cdot 4^{p-2}+4\cdot 3^{p-2}+3-2^p-2\cdot 3^p>0,
\end{align*}
since $2\cdot 4^{p-1}+9\cdot 4^{p-2}+3>4^p+3>2\cdot 3^p$ and $4\cdot 3^{p-2}\ge 2^p$.\\[1ex]
\indent(v) Going from $C$ to $C'$, we see that the degree of $v$ is increased by $t+1$, the degree of $x$ is decreased by $t+1$, the degree of $z$ is decreased from $t+1$ to $1$, and the degrees of $y_1,\dots,y_t$ are each decreased from $2$ to $1$. Since $d_C(v)\ge d_C(x)\ge t+1$, we have
\begin{align*}
e_p(C')-e_p(C) &= (d_{C}(v)+t+1)^p-d_C(v)^p+(d_{C}(x)-t-1)^p-d_C(x)^p\\
&\quad\quad\quad+1^p-(t+1)^p+t(1^p-2^p)\\
&>\sum_{1\le j\le p\textup{, $j$ odd}}{p\choose j}(d_C(v)^{p-j}-d_C(x)^{p-j})(t+1)^j\\
&\quad\quad\quad+{p\choose 2}(d_C(v)^{p-2}+d_C(x)^{p-2})(t+1)^2-(t+1)^p-t\cdot 2^p\\
&\ge 2(t+1)^p-(t+1)^p-t\cdot 2^p=(t+1)^p-t\cdot 2^p>0.
\end{align*}
%It follows that 
%\[
%e_p(C')-e_p(C)\ge 2(t+1)^{p-1}+(t+1)^p+(t+1)^{p-2}t^2-t^p-t\cdot 2^p>0,
%\]
%since $(t+1)^p>t^p$ and $2(t+1)^{p-1}+(t+1)^{p-2}t^2>(t+1)^p>t\cdot 2^p$.
\indent(iv) This follows from (v), since we can obtain the graph $C''$ from $C$ by adding the edge $xz$, so that $e_p(C)<e_p(C'')<e_p(C')$.
\end{proof}

\begin{prop}\label{Blsprp}
%Conjecture \ref{Blsconj} holds for $n>(\ell+s)^2$ if the following statement is true: Let $p\ge 2$, $\ell\ge 6$, and $s\ge 0$. Let $C$ be a connected $B_{\ell,s}$-free graph with $c\ge \ell+s$ vertices and $\Delta(C)\ge \ell+s-1$. Then $e_p(C)\le e_p(H(c,\ell))$, with equality if and only if $C=H(c,\ell)$.
%
Conjecture \ref{Blsconj} holds if the following statement is true: Let $p\ge 2$, $\ell\ge 6$, and $s\ge 0$. Then there exists $d=d(\ell,s)\ge\ell+s$ such that, for all connected $B_{\ell,s}$-free graph $C$ with $c\ge d=d(\ell,s)$ vertices and $\Delta(C)\ge d-1$, we have $e_p(C)\le e_p(H(c,\ell))$, with equality if and only if $C=H(c,\ell)$.

Similarly, Theorem \ref{B5sthm} holds if the following statement is true: Let $p\ge 2$ and $s\ge 0$. Then for all connected $B_{5,s}$-free graph $C$ with $c\ge s+5$ vertices and $\Delta(C)\ge s+4$, we have 
\begin{equation}\label{B5seq1}
e_p(C)\le
\left\{
\begin{array}{ll}
e_p(H(c,5)) & \textup{\emph{if $s=0$,}}\\
e_p(K_1+M_{c-1}) & \textup{\emph{if $s\ge 1$,}}
\end{array}
\right.
\end{equation}
with equality if and only if $C=H(c,5)$ for $s=0$, and $C=K_1+M_{c-1}$ for $s\ge 1$.
\end{prop}

\begin{proof}
%Suppose that the assertion in Proposition \ref{Blsprp} holds. We prove that Conjecture \ref{Blsconj} holds for $n>(\ell+s)^2$. Clearly the graph $H(n,\ell)$ is $B_{\ell,s}$-free. Now, let $G$ be a $B_{\ell,s}$-free graph on $n$ vertices and $G\neq H(n,\ell)$. Then the assertion of Conjecture \ref{Blsconj} follows if we can prove that $e_p(G)<e_p(H(n,\ell))$.
%
%Suppose first that $\Delta(G)\le \ell+s-2$. Then since $n>(\ell+s)^2$, we have 
%\begin{align}
%(n-1)^p &= (n-1)(n-1)^{p-1} >(n-1)(\ell+s-1)^{2p-2}\nonumber\\
%& > (n-1)[(\ell+s-2)^{2p-2}+(2p-2)(\ell+s-2)^{2p-3}]\nonumber\\
%& > (n-1)(\ell+s-2)^{2p-2}+(2p-2)(\ell+s-2)^{2p-1}\nonumber\\
%&> n(\ell+s-2)^p,\label{Blsprpeq1}
%\end{align}
%so that 
%\[
%e_p(G)\le n(\ell+s-2)^p<(n-1)^p<e_p(H(n,\ell)).
%\]
%
%Now, let $\Delta(G)\ge \ell+s-1$. We have $G=C_1\cup\cdots\cup C_t\cup G^\ast$ for some $t\ge 1$, where $C_1,\dots,C_t$ are the components of $G$ with maximum degree at least $\ell+s-1$, and $G^\ast$ is the remaining subgraph, so that $\Delta(G^\ast)\le \ell+s-2$. Let $c_i=|C_i|$. By the assertion in Proposition \ref{Blsprp}, for every $1\le i\le t$, we have 
%\begin{equation}
%e_p(C_i)\le e_p(H(c_i,\ell)),\label{Blsprpeq2}
%\end{equation}
%with equality if and only if $C_i=H(c_i,\ell)$. We apply (\ref{Blsprpeq2}) to every $C_i$, and then apply Lemma \ref{Blslem1}(b) repeatedly $t-1$ times, and finally Lemma \ref{Blslem2}(b), if $|G^\ast|\ge 1$. We find that $e_p(G)<e_p(H(n,\ell))$, since $G\neq H(n,\ell)$ by assumption.
%
Suppose that the first assertion in Proposition \ref{Blsprp} holds. We prove that Conjecture \ref{Blsconj} holds for $n>(\ell+s+d)^2$. Clearly the graph $H(n,\ell)$ is $B_{\ell,s}$-free. Now, let $G$ be a $B_{\ell,s}$-free graph on $n$ vertices and $G\neq H(n,\ell)$. Then the assertion of Conjecture \ref{Blsconj} follows if we can prove that $e_p(G)<e_p(H(n,\ell))$.

Suppose first that $\Delta(G)\le d-2$. Then since $n-1>(d+1)^2$, we have 
\begin{align*}
(n-1)^p &= (n-1)(n-1)^{p-1} > (n-1)(d+1)^{2p-2}\\
& > (n-1)[d^{2p-2}+(2p-2)d^{2p-3}]\\
& > (n-1)d^{2p-2}+2d^{2p-1}> n(d-2)^p,
\end{align*}
so that 
\begin{equation}
e_p(G)\le n(d-2)^p<(n-1)^p<e_p(H(n,\ell)).\label{Blsprpeq1}
\end{equation}
 
Now, let $\Delta(G)\ge d-1$. Let $G^\ast\subset G$ be the subgraph consisting of the components with maximum degree at most $d-2$, so that $\Delta(G^\ast)\le d-2$. We have $G=C_1\cup\cdots\cup C_t\cup G^\ast$ for some $t\ge 1$, where $C_1,\dots,C_t$ are the components of $G$ with maximum degree at least $d-1$. Let $c_i=|V(C_i)|\ge d$. By the assertion in Proposition \ref{Blsprp}, for every $1\le i\le t$, we have 
\begin{equation}
e_p(C_i)\le e_p(H(c_i,\ell)),\label{Blsprpeq2}
\end{equation}
with equality if and only if $C_i=H(c_i,\ell)$. We apply (\ref{Blsprpeq2}) to every $C_i$, and then apply Lemma \ref{Blslem1}(b) repeatedly $t-1$ times, and finally Lemma \ref{Blslem2}(b), if $|V(G^\ast)|>0$. We find that $e_p(G)<e_p(H(n,\ell))$, since $G\neq H(n,\ell)$ by assumption.

By a similar argument, using Lemmas \ref{Blslem1}(a) and \ref{Blslem2}(a), and setting $\ell=5$, $d=s+5$, we see that the second assertion implies Theorem \ref{B5sthm} for $n>(2s+10)^2$. Note that the analogous inequality to (\ref{Blsprpeq1}) would be
\[
e_p(G)\le n(s+3)^p<(n-1)^p<e_p(H(n,5))<e_p(K_1+M_{n-1}).
\]
\end{proof}

We are now ready to prove Theorems \ref{B5sthm}, \ref{B6sthm} and \ref{B7sthm}. The arguments in all three proofs are similar. In outline, it suffices to verify the statements in Proposition \ref{Blsprp} for $\ell=5,6,7$. Let $C$ be a connected $B_{\ell,s}$-free graph on $c$ vertices as defined in the proposition. We may assume that $C$ does not contain any of the pendent subgraphs, otherwise we may apply Lemma \ref{Blslem3} to obtain another $B_{\ell,s}$-free graph $C'$ with $e_p(C)<e_p(C')$, so that we could consider the argument for $C'$ instead of $C$. Under this assumption, we then show that $e_p(C)\le e_p(K_1+M_{c-1})$ for $\ell=5$, $s\ge 1$, and $e_p(C)\le e_p(H(c,\ell))$ otherwise. In each case, equality occurs if and only if $C$ is the corresponding extremal graph.

\begin{proof}[Proof of Theorems \ref{B5sthm}]
It suffices to verify the second statement in Proposition \ref{Blsprp}. Let $C$ be a $B_{5,s}$-free connected graph with $c\ge s+5$ vertices, and $v\in V(C)$ with $d_C(v)=\Delta(C)\ge s+4$. By Lemma \ref{Blslem3}, we may assume that $C$ does not contain an $x$-pendent edge $xy$ where $x,y\in V(C-v)$. Otherwise, we may delete $xy$ and add $vy$ to obtain the $B_{5,s}$-free graph $C'$ with $e_p(C)<e_p(C')$ and $d_{C'}(v)=\Delta(C')$, and consider the graph $C'$ instead of $C$.

For $i\ge 1$, let $V_i$ be the set of vertices of $C$ at distance $i$ from $v$. Note that $|V_1|=d_C(v)\ge s+4$. Also, we have the following properties.
\begin{enumerate}
\item[(i)] $V_i=\emptyset$ for $i\ge 3$.
\item[(ii)] $C[V_2]$ does not contain an edge.
\item[(iii)] Every vertex of $V_2$ has exactly one neighbour in $V_1$.  
\item[(iv)] $C[V_1]$ contains at most one edge if $s=0$, and $C[V_1]$ does not contain a copy of the path $P_3$ if $s\ge 1$.
\end{enumerate} 
Otherwise, suppose that (i) is false. Then we have a copy of $B_{5,s}$, where the path $P_5$ in $B_{5,s}$ is $x_3x_2x_1vy_1$ with $x_i\in V_i$ for $i=1,2,3$, $y_1\in V_1$, and the remaining $s$ vertices of the $B_{5,s}$ are all in $V_1\setminus\{x_1,y_1\}$. Properties (ii) to (iv) also hold for similar reasons. If $V_2\neq\emptyset$, then we must have an edge $xy\in E(C)$ with $x\in V_1$ and $y\in V_2$. It follows from (i) to (iii) that $xy$ is an $x$-pendent edge. Therefore, we may assume that $V_2=\emptyset$. From (iv), we can now easily see that $C\subset H(c,5)$ if $s=0$; and $C\subset K_1+M_{c-1}$ if $s\ge 1$, since $C[V_1]$ consists of independent edges and isolated vertices. Consequently (\ref{B5seq1}) holds, as well as the cases of equality.
\end{proof}

\begin{proof}[Proof of Theorem \ref{B6sthm}]
It suffices to verify the first statement in Proposition \ref{Blsprp} for $\ell=6$, with $d=s+6$. Let $C$ be a connected graph with $c\ge s+6$ vertices, and $v\in V(C)$ with $d_C(v)=\Delta(C)\ge s+5$. We may assume that $C$ does not contain an $x$-pendent edge or an $x$-pendent triangle, where $x\in V(C-v)$. Otherwise in either case, we may obtain the $B_{6,s}$-free graph $C'$ as described in Lemma \ref{Blslem3} with $e_p(C)<e_p(C')$ and $d_{C'}(v)=\Delta(C')$, and consider the graph $C'$ instead of $C$.

For $i\ge 1$, let $V_i$ be the set of vertices of $C$ at distance $i$ from $v$. Note that $|V_1|=d_C(v)\ge s+5$. Also, we have the following properties.
\begin{enumerate}
\item[(i)] $V_i=\emptyset$ for $i\ge 4$.
\item[(ii)] $C[V_i]$ does not contain a copy of the path $P_{5-i}$, for $i=1,2,3$.
\item[(iii)] Every vertex of $V_3$ has exactly one neighbour in $V_2$.  
\end{enumerate} 
Otherwise if any of (i) to (iii) is false, then we can easily find a copy of $B_{6,s}$ with centre $v$. By (i) to (iii), we may assume that $V_3=\emptyset$, otherwise we have an $x$-pendent edge $xy\in E(C)$ where $x\in V_2$ and $y\in V_3$. Next, suppose that we have an edge $yy'\in C[V_2]$. If $y$ and $y'$ have distinct neighbours in $V_1$, then we can again easily find a copy of $B_{6,s}$ with centre $v$ in $C$. It follows from (ii) with $i=2$ that $y$ and $y'$ must each have exactly one neighbour in $V_1$, which is a common neighbour $x\in V_1$, and therefore $C[\{x,y,y'\}]$ is an $x$-pendent triangle. Thus, we may assume that $C[V_2]$ does not contain an edge. Since no $x$-pendent edge $xy$ exists where $x\in V_1$, $y\in V_2$, this means that every vertex of $V_2$ must have at least two neighbours in $V_1$. This implies that any two vertices $y,y'\in V_2$ cannot have a common neighbour in $V_1$, otherwise we can again easily find a copy of $B_{6,s}$ in $C$. Therefore, if $V_2\neq\emptyset$ with $V_2=\{y_1,\dots,y_q\}$ for some $q\ge 1$, then for $1\le k\le q$, if $\Gamma_k\subset V_1$ is the set of neighbours of $y_k$ in $V_1$, we have $|\Gamma_k|\ge 2$, and the sets $\Gamma_k$ must be disjoint. Let $X=V_1\setminus\bigcup_{k=1}^q\Gamma_k$. Note that if we have an edge $e$ in $C[V_1]$, then $e$ must either be within $X$, or $e$ connects the two vertices of some $\Gamma_k$ with $|\Gamma_k|=2$, otherwise we can again find a copy of $B_{6,s}$ in $C$. Together with (ii) with $i=1$, we see that $C[V_1\cup V_2]$ does not contain a copy of the path $P_4$ (whether $V_2\neq\emptyset$ or $V_2=\emptyset$). Therefore, we can easily deduce that $C[V_1\cup V_2]$ is a subgraph whose components are stars and triangles.

Let $C^\ast$ be the graph obtained from $C$ by adding all edges from $v$ to $V_2$. Note that by replacing $C^\ast-v$ with the star of the same order, we obtain the graph $H(c,6)$. We shall show that this operation does not decrease the value of $e_p$. Consider the following operations.\\[1ex]
\indent(A) Suppose that $C^\ast-v$ contains two star components, say with centres $x$ and $y$, and the leaves at $y$ are $y_1,\dots,y_m$ for some $m\ge 0$. We delete the edges $yy_1,\dots,yy_m$ and add the edges $xy,xy_1,\dots,xy_m$. The increase in the value of $e_p$ is
\[
(d_{C^\ast}(x)+m+1)^p-d_{C^\ast}(x)^p+2^p-(m+1)^p>2^p>0.
\]

(B) Suppose that $C^\ast-v$ contains at least two triangle components, say with vertices $x_1,\dots,x_{3m}$ for some $m\ge 2$. We delete the edges of the triangles, and connect $x_1$ to $x_2,\dots,x_{3m}$. The increase in the value of $e_p$ is
\[
(3m)^p-3^p+(3m-1)(2^p-3^p)=(m^p-3m)3^p+(3m-1)2^p>0.
\]

(C) Suppose that $C^\ast-v$ contains a star and a triangle component, exactly one of each. Let $x$ be the centre of the star, and note that since $|V(C^\ast-v)|=c-1\ge s+5\ge 5$, we have $d_{C^\ast}(x)\ge 2$. We delete the edges of the triangle and connect $x$ to its three vertices. The increase in the value of $e_p$ is
\[
(d_{C^\ast}(x)+3)^p-d_{C^\ast}(x)^p+3(2^p-3^p).
\]
If $p=2$, then the increase is $6d_{C^\ast}(x)-6>0$. If $p\ge 3$ and $d_{C^\ast}(x)=2$, then the increase is $5^p+2^{p+1}-3^{p+1}>0$. Otherwise if $p\ge 3$ and $d_{C^\ast}(x)\ge 3$, then the increase is at least
\[
3pd_{C^\ast}(x)^{p-1}+3(2^p-3^p)\ge 3^{p+1}+3(2^p-3^p)>0.
\]

Therefore where possible, we apply operation (B), followed by successive applications of operation (A), and finally operation (C). We obtain $e_p(C)\le e_p(C^\ast)\le e_p(H(c,6))$. Equality occurs if and only if $C=C^\ast$ and $C^\ast-v$ is itself a star. That is, if and only if $C=H(c,6)$. 
\end{proof}

\begin{proof}[Proof of Theorem \ref{B7sthm}]
It suffices to verify the first statement in Proposition \ref{Blsprp} for $\ell=7$, with $d=2s+24$. Let $C$ be a connected graph with $c\ge 2s+24$ vertices, and $v\in V(C)$ with $d_C(v)=\Delta(C)\ge 2s+23$. We may assume that $C$ does not contain an $x$-pendent edge, triangle, diamond, spindle, or spindle$^+$, where $x\in V(C-v)$. Otherwise, we may obtain the $B_{7,s}$-free graph $C'$ as described in Lemma \ref{Blslem3} with $e_p(C)<e_p(C')$ and $d_{C'}(v)=\Delta(C')$, and consider the graph $C'$ instead of $C$.

For $i\ge 1$, let $V_i$ be the set of vertices of $C$ at distance $i$ from $v$. Note that $|V_1|=d_C(v)\ge 2s+23$. Also, we have the following properties.
\begin{enumerate}
\item[(i)] $V_i=\emptyset$ for $i\ge 5$.
\item[(ii)] $C[V_i]$ does not contain a copy of the path $P_{6-i}$, for $i=1,2,3,4$.
\item[(iii)] Every vertex of $V_4$ has exactly one neighbour in $V_3$. 
\end{enumerate} 
Otherwise if any of (i) to (iii) is false, then we can easily find a copy of $B_{7,s}$ with centre $v$. Proceeding exactly the same way as we did in Theorem \ref{B6sthm}, by avoiding a copy of $B_{7,s}$, or an $x$-pendent edge or triangle for some $x\in V(C-v)$, we can obtain the following facts.
\begin{itemize}
\item We may assume that $V_4=\emptyset$.
\item We may assume that $C[V_3]$ does not contain an edge, and that every vertex of $V_3$ has at least two neighbours in $V_2$. If $V_3\neq\emptyset$, with $V_3=\{y_1,\dots,y_q\}$ for some $q\ge 1$, and $\Gamma_k\subset V_2$ is the set of neighbours of $y_k$ in $V_2$, we have $|\Gamma_k|\ge 2$, and the sets $\Gamma_k$ must be disjoint. For $X=V_2\setminus\bigcup_{k=1}^q\Gamma_k$, if we have an edge $e$ in $C[V_2]$, then $e$ must either be within $X$, or $e$ connects the two vertices of some $\Gamma_k$ with $|\Gamma_k|=2$.
\end{itemize}
%By (i) to (iii), we may assume that $V_4=\emptyset$, otherwise we have an $x$-pendent edge $xy\in E(C)$ where $x\in V_3$ and $y\in V_4$. Next, suppose that we have an edge $yy'\in C[V_3]$. If $y$ and $y'$ have distinct neighbours in $V_2$, then we can again easily find a copy of $B_{7,s}$ with centre $v$ in $C$. It follows from (iv) that $y$ and $y'$ must each have exactly one neighbour in $C$, which is a common neighbour $x\in V_2$, and therefore $T(x,y,y')$ is an $x$-pendent triangle, a contradiction. Thus, we may assume that $C[V_3]$ does not contain an edge. Since no $x$-pendent edge $xy$ exists where $x\in V_2$, $y\in V_3$, this means that every vertex of $V_3$ must have at least two neighbours in $V_2$. This implies that any two vertices $y,y'\in V_3$ cannot have a common neighbour in $V_2$, otherwise we can again easily find a copy of $B_{7,s}$ in $C$. 

Now for any $\Gamma_k$, any two vertices $y,y'\in \Gamma_k$ cannot have two distinct neighbours in $V_1$, otherwise we can find a copy of $B_{7,s}$. Thus, the vertices of $\Gamma_k$ must have one common neighbour $x_k\in V_1$, so that $C[\Gamma_k\cup \{x_k,y_k\}]$ is either an $x_k$-pendent diamond or an $x_k$-pendent spindle. Therefore, we may further assume that $V_3=\emptyset$.

By (ii) with $i=2$, we see that the components of $C[V_2]$ are stars and triangles. Suppose that we have a star component in $C[V_2]$ with centre $z$ and leaves $y_1,\dots,y_t$, for some $t\ge 2$. Then no two of the $y_k$ can have distinct neighbours in $V_1$, otherwise we can find a copy of $B_{7,s}$. Thus, the vertices $y_k$ must have one common neighbour $x\in V_1$. If $z$ has a neighbour $x'\in V_1\setminus\{x\}$, then we have a copy of $B_{7,s}$ with centre $v$, where the $P_7$ is $y_1xy_2zx'vx''$ for some $x''\in V_1\setminus\{x,x'\}$, and the $s$ leaves are in $V_1\setminus\{x,x',x''\}$. Therefore, $x$ must be the unique neighbour of $z$ in $V_1$, and $C[\{x,z,y_1,\dots,y_t\}]$ is an $x$-pendent spindle$^+$.

Thus, we may assume that the components of $C[V_2]$ are triangles, and single edges and isolated vertices. We consider the behaviour of the edges that connect these components to $V_1$, keeping in mind that we should avoid creating a copy of  $B_{7,s}$. 
\begin{itemize}
\item If $y_1,y_2,y_3\in V_2$ form a triangle in $C[V_2]$, then $y_1,y_2,y_3$ must have a unique common neighbour in $V_1$, and they do not have any other neighbours in $V_1$.
\item Let $y_1y_2$ be a single edge in $C[V_2]$. If $y_1,y_2$ have exactly one common neighbour $x\in V_1$, then exactly one of $y_1,y_2$ has at least one  neighbour in $V_1\setminus\{x\}$, otherwise $C[\{x,y_1,y_2\}]$ is an $x$-pendent triangle or there is a copy of $B_{7,s}$. If $y_1,y_2$ have exactly two common neighbours $x_1,x_2\in V_1$, then both $y_1,y_2$ cannot have a neighbour in $V_1\setminus\{x_1,x_2\}$. Also, $y_1,y_2$ cannot have at least three common neighbours in $V_1$. The remaining possibility is that  $y_1, y_2$ have no common neighbour in $V_1$.
\item If $y$ is an isolated vertex in $C[V_2]$, then $y$ must have at least two neighbours in $V_1$, otherwise there is an $x$-pendent edge $xy$, for some $x\in V_1$.
\end{itemize}

Let $\tilde{C}=(C-v)-E(C[V_1])$, i.e., $\tilde{C}$ is the subgraph on $V_1\cup V_2$, with the edges of $C[V_1]$ deleted. Then, when considering the components of $\tilde{C}$, these components are the subgraphs as shown in Figure 2(a). We refer the subgraphs  illustrated as \emph{Type 1} to \emph{Type 5}.\\[1ex]
\[ \unit = 1cm
%(a)
\thnline 
\dl{0.2}{0}{-5.7}{0}\dl{0.2}{1}{-5.7}{1}
\bez{-5.7}{0}{-6.2}{0}{-6.2}{0.5}\bez{-5.7}{1}{-6.2}{1}{-6.2}{0.5}
\bez{0.2}{0}{0.7}{0}{0.7}{0.5}\bez{0.2}{1}{0.7}{1}{0.7}{0.5}
\dl{0.2}{1.5}{-5.7}{1.5}\dl{0.2}{2.5}{-5.7}{2.5}
\bez{-5.7}{1.5}{-6.2}{1.5}{-6.2}{2}\bez{-5.7}{2.5}{-6.2}{2.5}{-6.2}{2}
\bez{0.2}{1.5}{0.7}{1.5}{0.7}{2}\bez{0.2}{2.5}{0.7}{2.5}{0.7}{2}
%5
\pt{-0.2}{0.5}
\ellipse{-0.2}{1.85}{0.4}{0.15}
\point{-0.45}{2.15}{{\footnotesize$\ge 2$}}
\dl{-0.2}{0.5}{-0.6}{1.84}\dl{-0.2}{0.5}{0.2}{1.84}
%4
\pt{-1.4}{0.5}\pt{-2}{0.5}
\ellipse{-1.3}{1.85}{0.3}{0.15}\ellipse{-2.1}{1.85}{0.3}{0.15}
\point{-2.35}{2.15}{{\footnotesize$\ge 1$}}\point{-1.55}{2.15}{{\footnotesize$\ge 1$}}
\dl{-1.4}{0.5}{-2}{0.5}
\dl{-1.4}{0.5}{-1.6}{1.84}\dl{-1.4}{0.5}{-1}{1.84}\dl{-2}{0.5}{-2.4}{1.84}\dl{-2}{0.5}{-1.8}{1.84}
%3
\pt{-2.8}{0.5}\pt{-3.3}{0.5}
\pt{-2.8}{1.85}\pt{-3.3}{1.85}
\dl{-2.8}{0.5}{-3.3}{0.5}\dl{-2.8}{0.5}{-2.8}{1.85}\dl{-2.8}{0.5}{-3.3}{1.85}\dl{-3.3}{0.5}{-2.8}{1.85}\dl{-3.3}{0.5}{-3.3}{1.85}
%2
\pt{-4}{0.5}\pt{-4.5}{0.5}
\ellipse{-4}{1.85}{0.3}{0.15}\pt{-4.5}{1.85}
\point{-4.25}{2.15}{{\footnotesize$\ge 1$}}
\dl{-4}{0.5}{-4.5}{0.5}\dl{-4}{0.5}{-4.5}{1.85}\dl{-4.5}{0.5}{-4.5}{1.85}
\dl{-4}{0.5}{-4.3}{1.84}\dl{-4}{0.5}{-3.7}{1.84}
%1
\pt{-4.9}{0.4}\pt{-5.7}{0.4}\pt{-5.3}{0.6}
\pt{-5.3}{1.85}
\dl{-4.9}{0.4}{-5.7}{0.4}\dl{-4.9}{0.4}{-5.3}{0.6}\dl{-5.7}{0.4}{-5.3}{0.6}\dl{-4.9}{0.4}{-5.3}{1.85}\dl{-5.7}{0.4}{-5.3}{1.85}\dl{-5.3}{0.6}{-5.3}{1.85}
\point{-6.59}{1.88}{{\footnotesize$V_1$}}\point{-7.2}{2.28}{{\footnotesize (a)}}
\point{-6.59}{0.38}{{\footnotesize$V_2$}}\point{-6.59}{-0.4}{{\footnotesize Type:}}\point{-5.36}{-0.4}{{\footnotesize$1$}}\point{-4.17}{-0.4}{{\footnotesize$2$}}
\point{-3.12}{-0.4}{{\footnotesize$3$}}\point{-1.77}{-0.4}{{\footnotesize$4$}}\point{-0.27}{-0.4}{{\footnotesize$5$}}
%(b)
\thnline 
\dl{2.7}{0}{6.7}{0}\dl{2.7}{1}{6.7}{1}
\bez{2.7}{0}{2.2}{0}{2.2}{0.5}\bez{2.7}{1}{2.2}{1}{2.2}{0.5}
\bez{6.7}{0}{7.2}{0}{7.2}{0.5}\bez{6.7}{1}{7.2}{1}{7.2}{0.5}
\dl{2.7}{1.5}{6.7}{1.5}\dl{2.7}{2.5}{6.7}{2.5}
\bez{2.7}{1.5}{2.2}{1.5}{2.2}{2}\bez{2.7}{2.5}{2.2}{2.5}{2.2}{2}
\bez{6.7}{1.5}{7.2}{1.5}{7.2}{2}\bez{6.7}{2.5}{7.2}{2.5}{7.2}{2}
%1
\pt{2.7}{0.4}\pt{3.5}{0.4}\pt{3.1}{0.6}
\pt{4.3}{0.4}\pt{5.1}{0.4}\pt{4.7}{0.6}
\point{3.68}{0.41}{$\cdots$}
\pt{3.9}{2}
\dl{2.7}{0.4}{3.9}{2}\dl{3.5}{0.4}{3.9}{2}\dl{3.1}{0.6}{3.9}{2}\dl{4.3}{0.4}{3.9}{2}\dl{5.1}{0.4}{3.9}{2}\dl{4.7}{0.6}{3.9}{2}
\dl{2.7}{0.4}{3.5}{0.4}\dl{3.5}{0.4}{3.1}{0.6}\dl{3.1}{0.6}{2.7}{0.4}
\dl{4.3}{0.4}{5.1}{0.4}\dl{5.1}{0.4}{4.7}{0.6}\dl{4.7}{0.6}{4.3}{0.4}
%5
\pt{5.5}{0.5}\pt{5.9}{0.5}\pt{6.7}{0.5}
\point{6.08}{0.41}{$\cdots$}
\pt{5.9}{2}\pt{6.3}{2}
\dl{5.5}{0.5}{5.9}{2}\dl{5.9}{0.5}{5.9}{2}\dl{6.7}{0.5}{5.9}{2}\dl{6.3}{2}{5.5}{0.5}\dl{6.3}{2}{5.9}{0.5}\dl{6.3}{2}{6.7}{0.5}
\point{1.81}{0.38}{{\footnotesize$V_2$}}\point{1.2}{2.28}{{\footnotesize (b)}}
\point{1.81}{1.88}{{\footnotesize$V_1$}}
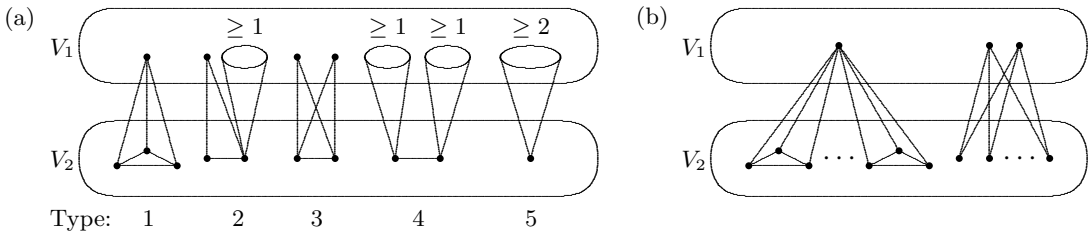
\ptlu{0}{-1.4}{\textup{\normalsize Figure 2. (a) Types 1 to 5 subgraphs; (b) How Type 1 and Type 5 subgraphs intersect}}
\]\\[0.7ex]
\indent For two such subgraphs $S,S'$, we see that, in order to avoid creating a $B_{7,s}$ with centre $v$, we would need to have $V(S)\cap V(S')\cap V_1=\emptyset$ in most cases. The only exceptions are when $S$ and $S'$ are of Type 1 and meeting at one vertex in $V_1$, or they are of Type 5 with order $3$ and meeting at exactly two vertices in $V_1$. Indeed, we may have at least two subgraphs meeting in $V_1$ in these two exceptional cases, as shown in Figure 2(b). We next eliminate these two possibilities.

Suppose first that we have $x\in V_1$, with exactly $m\ge 2$ Type 1 subgraphs meeting at $x$. Let $y_1,\dots,y_{3m}\in V_2$ be the vertices of the triangles in $V_2$. We delete the $4m$ edges of the Type 1 subgraphs, and add the edges $vy_1,\dots,vy_{3m}$. Then the degree of $v$ is increased by $3m$, the degree of $x$ is decreased by $3m$, and the degrees of the $y_k$ are decreased from $3$ to $1$. Since $d_C(v)\ge d_C(x)>3m$, the increase in the value of $e_p$ is
\begin{align*}
& (d_C(v)+3m)^p-d_C(v)^p+(d_C(x)-3m)^p-d_C(x)^p+3m(1^p-3^p)\\
>\:\:& \sum_{1\le j\le p\textup{, $j$ odd}}{p\choose j}(d_C(v)^{p-j}-d_C(x)^{p-j})(3m)^j\\
&\quad\quad\quad+{p\choose 2}(d_C(v)^{p-2}+d_C(x)^{p-2})(3m)^2-3m\cdot 3^p\\
>\:\:& 18m^2(3m)^{p-2}-m\cdot 3^{p+1}=2m^p\cdot 3^p-m\cdot 3^{p+1}\ge 4m\cdot 3^p-m\cdot 3^{p+1}>0.
\end{align*}

Next, suppose that we have $x_1,x_2\in V_1$, with exactly $m\ge 2$ Type 5 subgraphs of order $3$ meeting at  $x_1,x_2$. Let $y_1,\dots,y_m\in V_2$ be the vertices of these subgraphs in $V_2$. Note that the neighbours of $x_1$ (resp.~$x_2$) are precisely $v$ and the $y_k$, and possibly $x_2$ (resp.~$x_1$), otherwise there is a copy of $B_{7,s}$. Thus $d_C(x_1),d_C(x_2)\in \{m+1, m+2\}$. Suppose first that $m\le s+1$. Note that $d_C(v)\ge 2s+23>2(m+2)\ge d_C(x_1)+d_C(x_2)$. We delete the $2m$ edges of the Type 5 subgraphs, and add the edges $vy_1,\dots,vy_m$. Then the degree of $v$ is increased by $m$, the degrees of $x_1,x_2$ are decreased by $m$, and the degrees of the $y_k$ are decreased from $2$ to $1$. The increase in the value of $e_p$ is
\begin{align*}
& (d_C(v)+m)^p-d_C(v)^p+(d_C(x_1)-m)^p-d_C(x_1)^p+(d_C(x_2)-m)^p\\
&\quad\quad\quad -d_C(x_2)^p+m(1^p-2^p)\\
>\:\:& \sum_{1\le j\le p\textup{, $j$ odd}}{p\choose j}(d_C(v)^{p-j}-d_C(x_1)^{p-j}-d_C(x_2)^{p-j})m^j\\
&\quad\quad\quad +{p\choose 2}(d_C(v)^{p-2}+d_C(x_1)^{p-2}+d_C(x_2)^{p-2})m^2-m\cdot 2^p\\
>\:\:& ((2m)^{p-2}+2m^{p-2})m^2-m\cdot 2^p=(2^{p-2}+2)m^p-m\cdot 2^p> 0.
\end{align*}

Secondly, let $m\ge s+2$. Suppose that there is an edge $x'y'$ where either $x'\in V_1\setminus\{x_1,x_2\}$ and $y'\in (V_1\cup V_2)\setminus\{x_1,x_2,y_1,\dots,y_m\}$, or $x'=x_2$ and $y'\in V_1\setminus\{x_1,x_2\}$. Then there is a copy of $B_{7,s}$ with centre $x_1$, where the $P_7$ is $y'x'vx_2y_2x_1y_1$ or $wvy'x_2y_2x_1y_1$ for some $w\in V_1\setminus\{x_1,x_2,y'\}$, and the $s$ leaves are $y_3,\dots,y_{s+2}$. Similarly, we cannot have an edge $x_1y'$ for every $y'\in V_1\setminus\{x_1,x_2\}$. It follows that all the edges of $C$ are those connecting $v$ to $V_1$, and all edges between $\{x_1,x_2\}$ and $\{y_1,\dots,y_m\}$, and possibly $x_1x_2$. Now, let $C'$ be the graph obtained by deleting the edges $x_2y_1,\dots,x_2y_m$ and adding the edges $vy_1,\dots,vy_m$. Since $d_C(v)\ge d_C(x_2)$, the increase in the value of $e_p$ is
\begin{align*}
& (d_C(v)+m)^p-d_C(v)^p+(d_C(x_2)-m)^p-d_C(x_2)^p\\
\ge\:\:& \sum_{1\le j\le p\textup{, $j$ odd}}{p\choose j}(d_C(v)^{p-j}-d_C(x_2)^{p-j})m^j+{p\choose 2}(d_C(v)^{p-2}+d_C(x_2)^{p-2})m^2> 0.
\end{align*}
Moreover, we see that the degree sequence of $C'$ is majorised by the degree sequence of $K_2+E_{c-2}$, by identifying $\{v,x_1\}$ with $K_2$, and the remaining vertices of $C'$ with $E_{c-2}$. It follows that $e_p(C)<e_p(C')\le e_p(K_2+E_{c-2})<e_p(H(c,7))$.

Therefore, we may assume that no two of the subgraphs as shown in Figure 2(a) meet in $V_1$. For such a subgraph $S$, let $[S]$ denote the component of $C-v$ containing $S$. We consider the structure of $[S]$, so as to avoid a copy of $B_{7,s}$. Clearly if $S$ is of Type 1, then $[S]=S$. If $S$ is of Type 2, 3 or 4, then either $[S]=S$, or $|V(S)\cap V_1|=2$, and the edge connecting the two vertices of $V(S)\cap V_1$ is in $[S]$. Finally, let $S$ be of Type 5, with $V(S)\cap V_1=\{x_1,\dots,x_t\}$ for some $t\ge 2$. It is easy to check that any additional vertices and edges in $[S]$ are as follows. If $t=2$, then we may possibly have the edge $x_1x_2$, and either $x_1$ and $x_2$ are connected to another vertex of $V_1\setminus\{x_1,x_2\}$, or only one of $x_1,x_2$ is connected to some other vertices of $V_1\setminus\{x_1,x_2\}$. If $t=3$, then we may possibly have any number of the edges $x_1x_2, x_1x_3,x_2x_3$, or none of these three edges and only one of $x_1,x_2,x_3$ is connected to some other vertices of $V_1\setminus\{x_1,x_2,x_3\}$. If $t\ge 4$, then we may either have exactly one edge in $\{x_1,\dots,x_t\}$, or no edge in $\{x_1,\dots,x_t\}$ and only one of $x_1,\dots,x_t$ is connected to some other vertices of $V_1\setminus\{x_1,\dots,x_t\}$.

We see that all such components $[S]$ can be classified into exactly one of three types:
\begin{enumerate}
\item[(I)] A subgraph of $K_4$.
\item[(II)] A $H(c',5)$ for some $5\le c'\le c-1$ (i.e., a star on $c'$ vertices with an edge connecting two leaves).
\item[(III)] A double star with at least five vertices (i.e., two disjoint stars with an edge connecting their centres). A star itself is a special case of a double star.
\end{enumerate}
Moreover, by (ii) with $i=1$, we see that if $Y=V_1\setminus\bigcup V([S])$, where the union is taken over all such subgraphs $S$ in Figure 2(a), then $C[Y]$ is $P_5$-free. It is easy to show that the components of $C[Y]$ must also be one of the types (I), (II) or (III).
% Indeed, if $T\subset C[Y]$ is such a component, then either $T\subset K_4$, or $T$ has a vertex $z$ with degree at least $4$, and we may either have one edge connecting two neighbours of $z$, or one neighbour of $z$ has further neighbours that are non-adjacent to $z$.
Consequently, if we connect $v$ to all vertices of $V_2$ to obtain the graph $C^\ast$, then the components of $C^\ast-v$ are of the types (I), (II) or (III). Note that by replacing $C^\ast-v$ with the graph $H(c-1,5)$, we obtain the graph $H(c,7)$. We shall show that this operation does not decrease the value of $e_p$. Consider the following operations.\\[1ex]
\indent (A) Suppose that $C^\ast-v$ has a double star component with at least five vertices, which is not a star. Let the centres be $x,y$, and the leaves at $y$ be $y_1,\dots,y_m$, for some $m\ge 1$. We may assume that $d_{C^\ast}(x)\ge d_{C^\ast}(y)=m+2$. We obtain the star with the same order by deleting the edges $yy_1,\dots,yy_m$, and adding the edges $xy_1,\dots,xy_m$. If $m\ge 2$, then the increase in the value of $e_p$ is
\begin{align*}
(d_{C^\ast}(x)+m)^p-d_{C^\ast}(x)^p+2^p-(m+2)^p &> pd_{C^\ast}(x)^{p-1}m-(m+2)^p\\
&\ge 2(m+2)^{p-1}m-(m+2)^p\ge 0.
\end{align*}
If $m=1$, then we have $d_{C^\ast}(x)\ge 4$. In this case, the increase in the value of $e_p$ is
\[
(d_{C^\ast}(x)+1)^p-d_{C^\ast}(x)^p+2^p-3^p > pd_{C^\ast}(x)^{p-1}+2^p-3^p\ge 2\cdot 4^{p-1}+2^p-3^p>0.
\]
%***It follows that, if we have a star or a double star component with $c'\ge 5$ vertices in $C^\ast-v$, we can strictly increase the value of $e_p$ by replacing the component with $H(c',5)$. We may therefore assume that no such component exists in $C^\ast[V_1]$.***
\indent (B) Suppose that we have two components $C_1,C_2\subset C^\ast-v$ with $c_1$ and $c_2$ vertices, where $c_1\ge c_2\ge 5$, and $C_1$ (resp.~$C_2$) is either a star or the graph $H(c_1,5)$ (resp.~$H(c_2,5)$). If $C_1$ is a star, we add an edge to create $H(c_1,5)$, and likewise for $C_2$, so that we have the graphs $H(c_1,5)$ and $H(c_2,5)$. We then delete all edges of the $H(c_2,5)$, and connect all of its vertices to the universal vertex of the $H(c_1,5)$, thus obtaining the subgraph $H(c_1+c_2,5)$. The increase in the value of $e_p$ is at least
\[
(c_1+c_2)^p-c_1^p+2^p-c_2^p+2(2^p-3^p)>pc_1^{p-1}c_2-2\cdot 3^p> 0.
\]
%***Therefore, we may assume that there is at most one component $H(c_1,5)$ in $C^\ast[V_1]$ with $5\le c_1\le c-1$. ***

Let $R$ be the subgraph of $C^\ast-v$ consisting of the components which are subgraphs of $K_4$. We have $d_{C^\ast}(y)\le 4$ for all $y\in V(R)$. Let $|V(R)|=r$.\\[1ex]
\indent (C) Suppose $r\ge 16$. We replace $R$ with the star of order $r$, with centre $x\in V(R)$. The increase in the value of $e_p$ is
\[
r^p-d_{C^\ast}(x)^p+\sum_{y\in V(R-x)}(2^p-d_{C^\ast}(y)^p) > r^p-r\cdot 4^p\ge r(16^{p-1}-4^p)\ge 0.
\]

(D) Suppose that $1\le r\le 15$, and the subgraph $C^\ast-(\{v\}\cup V(R))$ is $H(c_1,5)$. Recall that $|V(C^\ast-v)|=c-1\ge 2s+23\ge 23$, and thus $c_1\ge 8$. We delete all edges of $R$, and connect all vertices of $R$ to the universal vertex of the $H(c_1,5)$, to form a copy of $H(c-1,5)$. Since $c_1+r=c-1$, the increase in the value of $e_p$ is
\begin{align*}
(c-1)^p-c_1^p+\sum_{y\in V(R)}(2^p-d_{C^\ast}(y)^p) &\ge (c_1+r)^p-c_1^p+r(2^p-4^p)\\
&> pc_1^{p-1}r-r\cdot 4^p\ge r(2\cdot 8^{p-1}-4^p)\ge 0.
\end{align*}

Therefore where possible, we apply operation (C), then apply operation (A) to all double stars in $C^\ast[V_1]$, followed by successive applications of operation (B), and finally operation (D). We obtain $e_p(C) \le e_p(C^\ast) \le e_p(H(c,7))$. Equality occurs if and only if $C = C^\ast$ and $C^\ast-v$ is the graph $H(c-1,5)$. That is, if and only if $C = H(c,7)$.

The proof of Theorem \ref{B7sthm} is now complete.
\end{proof}

\section*{Acknowledgements}
Yongxin Lan, Zhongmei Qin, and Yongtang Shi are partially supported by National Natural Science Foundation of China (Nos.~11371021, 11771221), and Natural Science Foundation of Tianjin (No.~17JCQNJC00300). Henry Liu is partially supported by the Startup Fund of One Hundred Talent Program of SYSU. Henry Liu would also like to thank the Chern Institute of Mathematics, Nankai University, for their generous hospitality. He was able to carry out part of this research during his visit there.

\end{document}